\documentclass[12pt]{amsart}

\usepackage{amsmath,amssymb}
\usepackage{bm}
\usepackage{graphicx}
\usepackage{ascmac}
\usepackage{amsthm}
\usepackage{amsfonts}
\usepackage{mathrsfs}
\usepackage{mathtools}
\usepackage[all]{xy}
\usepackage{comment}
\usepackage{appendix}

\theoremstyle{plain}
\newtheorem{thm}{Theorem}[section]
\newtheorem{cor}[thm]{Corollary}
\newtheorem{prop}[thm]{Proposition}
\newtheorem{lem}[thm]{Lemma}

\newtheorem*{claima}{Claim (a)}
\newtheorem*{claimb}{Claim (b)}

\theoremstyle{definition}
\newtheorem{dfn}[thm]{Definition}
\newtheorem{ex}[thm]{Example}
\newtheorem{rem}[thm]{Remark}

\newtheorem*{nota}{Notation}

\newtheorem*{oftp}{Organization of this paper}
\newtheorem*{ack}{Acknowledgments}
\renewcommand{\qed}{\hfill$\blacksquare$\par}

\numberwithin{thm}{section}
\numberwithin{equation}{section}

\newcommand{\Zpn}{\mathbb{Z}_{>0}}
\newcommand{\Znn}{\mathbb{Z}_{\geq 0}}
\newcommand{\Z}{\mathbb{Z}}

\newcommand{\Q}{\mathbb{Q}}

\newcommand{\G}{\mathbb{G}}

\newcommand{\sS}{\mathscr{S}}

\renewcommand{\hbar}{\overline{h}}

\renewcommand{\P}{\mathbb{P}}

\newcommand{\bK}{\mathbf{K}}

\newcommand{\bd}{\mathbf{d}}

\newcommand{\cH}{\mathcal{H}}

\newcommand{\sD}{\mathscr{D}}
\newcommand{\sP}{\mathscr{P}}

\DeclareMathOperator{\Ker}{Ker}
\DeclareMathOperator{\Ima}{Im}

\DeclareMathOperator{\Hom}{Hom}

\DeclareMathOperator{\id}{id}
\DeclareMathOperator{\Aut}{Aut}

\DeclareMathOperator{\ord}{ord}
\DeclareMathOperator{\Gal}{Gal}

\DeclareMathOperator{\N}{N}

\DeclareMathOperator{\red}{red}

\DeclareMathOperator{\Res}{Res}

\DeclareMathOperator{\rk}{rank}

\DeclareMathOperator{\Pic}{Pic}

\DeclareMathOperator{\sep}{sep}

\DeclareMathOperator{\pr}{pr}

\DeclareMathOperator{\tor}{tor}

\DeclareMathOperator{\fl}{fl}

\DeclareMathOperator{\nor}{nor}
\DeclareMathOperator{\set}{set}
\DeclareMathOperator{\srd}{srd}

\usepackage[OT2,T1]{fontenc}
\DeclareSymbolFont{cyrletters}{OT2}{wncyr}{m}{n}
\DeclareMathSymbol{\Sha}{\mathalpha}{cyrletters}{"58}

\usepackage{geometry}
\usepackage{xcolor}

\title[Rationality problem, I\hspace{-1pt}I]{The rationality problem for multinorm one tori, I\hspace{-1pt}I}

\author[S.~Hasegawa]{Sumito Hasegawa}
\address{(Hasegawa)}
\email{sumito.hasegawa@gmail.com}

\author[K.~Kanai]{Kazuki Kanai}
\address{(Kanai) Department of Computer and Information Sciences, College of Engineering, Ibaraki University, 4-12-1, Nakanarusawa-cho, Hitachi, Ibaraki, 316-8511, Japan}
\email{kazuki.kanai.du62@vc.ibaraki.ac.jp}

\author[Y.~Oki]{Yasuhiro Oki}
\address{(Oki) Department of Mathematics, College of Science, Rikkyo University, 3-34-1, Nishi-ikebukuro, Toshima-ku, Tokyo, 171-8501, Japan}
\email{oki@rikkyo.ac.jp}


\begin{document}

\begin{abstract}
We investigate the stable and retract rationality of multinorm one tori associated to finite {\'e}tale algebras.
Our results are organized according to the greatest common divisor $d$ of the degrees of the factors. 
We show that these tori are stably rational for $d=1$, and obtain a criterion for retract rationality that can be attributed to our previous results.
For $d>1$, we provide sufficient conditions for the failure of retract rationality. 
We further generalize results of Endo--Miyata (1975) and Endo (2011) by giving an equivalent condition for multinorm one tori to be stably rational under the assumption that they split over Galois extensions with Galois groups in which all Sylow subgroups are cyclic. 
A similar result also holds when they split over dihedral Galois extensions. 
\end{abstract}

\maketitle
\tableofcontents

\section{Introduction}

Let $k$ be a field, and fix a separable closure $k^{\sep}$ of $k$. Consider an algebraic variety $X$ over $k$.
We say that $X$ is \emph{rational} over $k$ if it is birationally equivalent to a projective space over $k$. 
It is also important to investigate weaker conditions than rationality, such as \emph{stable rationality}, \emph{retract rationality}, and \emph{unirationality}.
These satisfy the following:
\begin{center}
rational $\Rightarrow$ stably rational $\Rightarrow$ 
retract rational $\Rightarrow$ unirational. 
\end{center}
We remark that the direction of the implication cannot be reversed in general.
See Section \ref{sect:rvrt} for more details.

This paper follows \cite{Hasegawa}, which gives a foundation for the study of stable rationality and the retract rationality of multinorm one tori. Let $\bK$ be a finite {\'e}tale algebra over $k$, that is, a finite product of finite separable field extensions of $k$ contained in $k^{\sep}$. Then, we set
\begin{equation*}
    T_{\bK/k}:=\Ker(\N_{\bK/k}\colon \Res_{\bK/k}\G_{m}\rightarrow \G_{m}),
\end{equation*}
where $\Res_{\bK/k}$ is the Weil restriction. We call it the \emph{multinorm one torus} associated to $\bK/k$. It is an algebraic torus of dimension $[\bK:k]-1$ over $k$. Moreover, it splits over the Galois closure $L$ over $k$ of the composite field of all the factors of $\bK$, that is, there is an isomorphism of $L$-algebraic groups $T_{\bK/k}\otimes_{k}L \cong \G_{m,L}^{[\bK:k]-1}$. If $\bK$ is a field, $T_{\bK/k}$ is called a \emph{norm one torus}, and the stable (resp.~retract) rationality has been studied extensively in the literature. See the introduction in \cite{Hasegawa} for details. 

For a finite {\'e}tale algebra $\bK=\prod_{i=1}^{r}K_{i}$ over $k$, we define a positive integer $d_{k}(\bK)$ as follows: 
\begin{equation*}
d_{k}(\bK):=\gcd([K_{1}:k],\ldots,[K_{r}:k]). 
\end{equation*}
In this paper, we study the relation between the stable (resp.~retract) rationality of $T_{\bK/k}$ and $d_{k}(\bK)$. 
If $d_{k}(\bK)=1$, then we can completely resolve the problem as follows.  

\begin{thm}\label{mtd1}
Let $k$ be a field, and $\bK=\prod_{i=1}^{r}K_{i}$ be a finite {\'e}tale algebra over $k$, where $r\geq 2$, which satisfies $d_{k}(\bK)=1$. Then the multinorm one torus $T_{\bK/k}$ is stably rational over $k$. 
\end{thm}

In the case $d_{k}(\bK)>1$, the situation becomes more complicated compared to the above. However, we can give some sufficient conditions for $T_{\bK/k}$ to be not retract rational. 

\begin{thm}\label{mtd2}
Let $k$ be a field, and $\bK=\prod_{i=1}^{r}K_{i}$ a finite {\'e}tale algebra over $k$, where $r\geq 2$, which satisfies $d_{k}(\bK)=2$. For each $i$, we denote by $L_{i}$ the Galois closure of $K_{i}$ over $k$. 
\begin{enumerate}
\item If $r=2$, we further assume that 
    \begin{itemize}
    \item $[K_{1}:k]\in 4\Z$ or $[K_{2}:k]\in 4\Z$; and 
    \item $L_{1}\cap K_{2}=k$ or $L_{2}\cap K_{1}=k$.
    \end{itemize}
Then the multinorm one torus $T_{\bK/k}$ is not retract rational over $k$. 
\item If $r\geq 3$, we further assume that there is a subset $I$ of $\{1,\ldots,r\}$ with $\#I=3$ such that $L_{i}\cap K_{j}=k$ or $L_{j}\cap K_{i}=k$ for any $i,j\in I$ with $i\neq j$. Then the multinorm one torus $T_{\bK/k}$ is not retract rational over $k$. 
\end{enumerate}
\end{thm}

\begin{thm}\label{mtd3}
Let $k$ be a field, and $\bK=\prod_{i=1}^{r}K_{i}$ a finite {\'e}tale algebra over $k$, where $r\geq 2$, that satisfies $d_{k}(\bK)\geq 3$. We further assume $L_{i}\cap K_{j}=k$ for some $i\neq j$, where $L_{i}$ is the Galois closure of $K_{i}$ over $k$. Then the multinorm one torus $T_{\bK/k}$ is not retract rational over $k$. 
\end{thm}

\begin{rem}
Let $k$ be a field, and $\bK$ a finite {\'e}tale algebra over $k$. Choose a smooth compactification $X$ of $T_{\bK/k}$ over $k$, and put $\overline{X}:=X\otimes_{k}k^{\sep}$. Then, we have 
\begin{equation}\label{eq:hpic}
H^{1}(k,\Pic(\overline{X}))=0
\end{equation}
if $T_{\bK/k}$ is retract rational over $k$. On the other hand, \eqref{eq:hpic} may hold even if $T_{\bK/k}$ is not retract rational. In \cite[Theorem 1]{Demarche2014}, Demarche and Wei essentially proved \eqref{eq:hpic} in the case $L_{I}\cap K_{J}=k$. Here $I$ and $J$ are subsets of $\{1,\ldots,r\}$ with $I\sqcup J=\{1,\ldots,r\}$, $L_{I}$ is the composite field of $L_{i}$ for all $i\in I$ and $K_{J}$ is the composite field of $K_{j}$ for all $j\in J$. Pick $i_{0}\in I$ and $j_{0}\in J$. Then it is clear that $L_{i_{0}}\cap K_{j_{0}}=k$ is true. In particular, if $d_{k}(\bK)\geq 3$, then the assumption in Theorem \ref{mtd3} is much weaker than the one in \cite[Theorem 1]{Demarche2014}. 
\end{rem}

We also discuss some cases where the assumption in Theorem \ref{mtd2} or Theorem \ref{mtd3} fails. We first consider the case that $T_{\bK/k}$ splits over a finite Galois extension of which all Sylow subgroups of the Galois group are cyclic. In this case, \cite[Theorem 1.5]{Endo1975} implies that $T_{\bK/k}$ is always retract rational over $k$. 

\begin{thm}\label{mth1}
Let $k$ be a field, and $\bK=\prod_{i=1}^{r}K_{i}$ a finite {\'e}tale algebra over $k$ satisfying $K_{i}\not\subset K_{j}$ for any $i \neq j$. Denote by $L/k$ the Galois closure of the composite of $K_{1},\ldots,K_{r}$ over $k$, and put $G:=\Gal(L/k)$. We further assume that all Sylow subgroups of $G$ are cyclic. Then the $k$-torus $T_{\bK/k}$ is retract rational. Moreover, it is stably rational over $k$ if and only if 
\begin{equation*}
\sD_{k}(\bK)\cap \sP_{*}(G)=\emptyset. 
\end{equation*}
Here $\sD_{k}(\bK)$ is the set of prime divisors of $d_{k}(\bK)$, and $\sP_{*}(G)$ is the set of prime divisors $p$ of $\#G$ satisfying two conditions as follows: 
\begin{itemize}
\item a $p$-Sylow subgroup $G_{p}$ of $G$ is normal; and
\item the image of the homomorphism $G\rightarrow \Aut(G_{p})$ induced by conjugation has order $\geq 3$. 
\end{itemize}
\end{thm}

If $\bK=K$ is a field, then Theorem \ref{mth1} follows from \cite[Theorem 2.3]{Endo1975} and \cite[Theorem 3.1]{Endo2011}. 

\begin{rem}
The set $\sP_{*}(G)$ is empty if $G$ is cyclic. In particular, we obtain that $T_{\bK/k}$ is stably rational over $k$ if $L/k$ is cyclic. 
\end{rem}

Second, we argue the case where $T_{\bK/k}$ splits over finite dihedral Galois extensions. Denote by $D_{n}$ the dihedral group of order $2n$, that is, 
\begin{equation*}
    D_{n}:=\langle \sigma_{n},\tau_{n}\mid \sigma_{n}^{n}=\tau_{n}^{2}=1,\tau_{n}\sigma_{n}\tau_{n}^{-1}=\sigma_{n}^{-1}\rangle. 
\end{equation*}
Note that there is an isomorphism $D_{2}\cong (C_{2})^{2}$. 

\begin{thm}\label{mth2}
Let $k$ be a field, and $\bK=\prod_{i=1}^{r}K_i$ a finite {\'e}tale algebra over $k$ that satisfies $K_{i}\not\subset K_{j}$ for any $i \neq j$. Denote by $L/k$ the Galois closure of the composite of $K_{1},\ldots,K_{r}$ over $k$. We further assume that there is an isomorphism $\Gal(L/k)\cong D_{n}$ for some $n\in \Zpn$. 
\begin{enumerate}
\item If $d_{k}(\bK)$ is odd, then $T_{\bK/k}$ is stably rational over $k$. 
\item If $d_{k}(\bK)$ is even, then $T_{\bK/k}$ is stably rational over $k$ if and only if it is retract rational over $k$. Moreover, the above two equivalent conditions hold if and only if there exists a finite product $\bK'=\prod_{i=1}^{r'}K'_{i}$ of intermediate fields $K'_{1},\ldots,K'_{r'}$ of $L/k$ such that 
\begin{equation*}
    T_{\bK/k}\times_{k}S\cong T_{\bK'/k}\times_{k}S'
\end{equation*}
for some quasi-trivial $k$-tori $S$ and $S'$. Here, $d_{k}(\bK)$ and $\bK'$ satisfy one of the following: 
\begin{itemize}
\item[(a)] $r'=1$, $d_{k}(\bK)\in 2\Z \setminus 4\Z$, $d_{k}(\bK)\mid 2n$ and $\Gal(L/K'_{1})\cong \langle \sigma_{n}^{d_{k}(\bK)/2}\rangle$; 
\item[(b)] $r'=1$, $d_{k}(\bK)\in 2\Z \setminus 4\Z$, $d_{k}(\bK)\mid n$ and $\Gal(L/K'_{1})\cong \langle \sigma_{n}^{d_{k}(\bK)},\tau_{n}\rangle$; 
\item[(c)] $r'=2$, $d_{k}(\bK)\in 2\Z \setminus 4\Z$, $d_{k}(\bK)\mid n$, $\Gal(L/K'_{1})\cong \langle \sigma_{n}^{d_{k}(\bK)/2}\rangle$ and $\Gal(L/K'_{2})\cong \langle \sigma_{n}^{d_{k}(\bK)},\tau_{n}\rangle$;
\item[(d)] $r'=2$, $d_{k}(\bK)\mid n$, $\Gal(L/K'_{1})\cong \langle \sigma_{n}^{d_{k}(\bK)},\tau_{n}\rangle$ and $\Gal(L/K'_{2})\cong \langle \sigma_{n}^{d_{k}(\bK)},\sigma_{n}\tau_{n}\rangle$. 
\end{itemize}
In each of the conditions \textup{(a)--(d)}, the subgroup isomorphisms
are understood to be restrictions of a common isomorphism
$\Gal(L/k)\cong D_n$.
\end{enumerate}
\end{thm}

If $\bK$ is a field, Theorem \ref{mth2} follows from \cite[Theorem 1.5, Theorem 2.3]{Endo1975}, \cite[Theorem 3.1]{Endo2011} and \cite[Theorem 1.1]{Hoshi2024}. Our contribution to Theorem \ref{mth2} is the case that $\bK$ is not a field. 

\vspace{6pt}

We now explain our proof of main theorems. It is based on the standard method, which rephrases the stable (resp.~retract) rationality of algebraic tori by means of their character groups. For an algebraic torus $T$ over $k$, let
\begin{equation*}
    X^{*}(T):=\Hom_{k^{\sep}\text{-groups}}(T\otimes_{k}k^{\sep},\G_{m,k^{\sep}}). 
\end{equation*}
Let $L/k$ be a finite Galois extension that contains $K_{1},\ldots,K_{r}$. Then, the Galois group $\Gal(L/k)$ acts on $X^{*}(T_{\bK/k})$. Moreover, it is known that $T$ is stably (resp.~retract) rational over $k$ if and only if $X^{*}(T)$ is quasi-permutation (resp.~quasi-invertible). 

For Theorem \ref{mtd1}, we prove that $X^{*}(T_{\bK/k})\oplus \Z$ is a permutation $\Gal(L/k)$-lattice. Moreover, we obtain the following as an application of Theorem \ref{mtd1}. 

\begin{thm}\label{mth3}
Let $k$ be a field, $\bK=\prod_{i=1}^{r}K_{i}$ a finite {\'e}tale algebra over $k$. Take a finite Galois extension $L$ of $k$ containing $K_{1},\ldots,K_{r}$. For each prime divisor $p$ of $d_{k}(\bK)$, pick a maximal subextension $L^{(p)}$ of $L/k$ such that its extension degree is coprime to $p$. Then, the following are equivalent: 
\begin{enumerate}
    \item $T_{\bK/k}$ is retract rational over $k$; 
    \item $T_{\bK/k}$ is retract rational over $L^{(p)}$ for any prime divisor $p$ of $d_{k}(\bK)$. 
\end{enumerate}
\end{thm}

\begin{rem}\label{rem:rdct}
The condition (ii) in Theorem \ref{mth3} can be determined by the previous result (\cite[Theorems 1.2, 1.3]{Hasegawa}). Therefore, we obtain a criterion for arbitrary multinorm one tori to be retract rational.
\end{rem}

Theorems \ref{mtd2} and \ref{mtd3} follow from the criterion in Remark \ref{rem:rdct}. Next, we give an outline of the proof of Theorem \ref{mth1}. 
The field case $\bK=K$ was established by \cite{Endo1975} and \cite{Endo2011}.
Hence it remains to treat the case $r\geq 2$, that is, $\bK$ is not a field.
Using a technique introduced by \cite[\S 3, \S 8]{Hasegawa}, 
which appropriately replaces {\'e}tale algebras, we reduce the desired assertion to the field case.
The technique mentioned above is also used in the proof of Theorem \ref{mth2}, and in particular we may assume that $d_{k}(\bK)$ equals $n$ or $2n$. Then, (i) is a consequence of Theorem \ref{mth1}. Regarding (ii), except for the part concerning the stable rationality of (c), this can be resolved by combining previous research, including \cite{Hasegawa}. Finally, the remaining part, that is, $T_{\bK'/k}$ is stably rational provided that $\bK'$ satisfies (c), is proved by specifically implementing a coflabby resolution of $X_{*}(T_{\bK'/k})$. 

\begin{oftp}
In Section \ref{sect:rvrt}, we recall some basic theory on algebraic tori and their character groups that will be used in this paper. 
Next, we restrict our attention to multinorm one tori in Section \ref{sect:mnot}. More precisely, we list some general facts from \cite{Hasegawa}. 
All sections from Section \ref{sect:pf13} onward are proofs of the main theorems. 
We first prove Theorems \ref{mtd1} and \ref{mth3} in Section \ref{sect:pf13}. 
Second, in Section \ref{sect:pf23}, proofs of Theorems \ref{mtd2} and \ref{mtd3} are provided. 
Furthermore, we conclude the proof of Theorem \ref{mth1} in Section \ref{sect:pfzg}. 
Section \ref{sect:cflb} is the most technical part of this paper, which is devoted to proving that $T_{\bK'/k}$ is stably rational when $\bK'$ satisfies (c) in Theorem \ref{mth2}. 
Finally, we complete a proof of Theorem \ref{mth2} in Section \ref{sect:pfdn} using the results from the previous sections.
\end{oftp}

\begin{ack}
The third-named author was supported by the JSPS KAKENHI Grant Number JP24K16884.
\end{ack}

\begin{nota}
Let $G$ be a finite group. 
\begin{itemize}
\item For a subgroup $H$ of $G$, we define $N^{G}(H)$ and $N_{G}(H)$ as follows: 
\begin{equation*}
N^{G}(H):=\bigcap_{g\in G}gHg^{-1},\quad
N_{G}(H):=\{g\in G\mid gHg^{-1}=H\}. 
\end{equation*}
Note that $N^{G}(H)$ and $N_{G}(H)$ are called the \emph{normal core} of $H$ and the \emph{normalizer} of $H$ in $G$, respectively. 
\item A $G$-lattice refers to a finitely generated free abelian group equipped with a left action of $G$. For a $G$-lattice $M$, the dual lattice of $M$ is denoted by
\begin{equation*}
    M^{\circ}:=\Hom_{\Z}(M,\Z). 
\end{equation*}
Here we define a left action of $G$ on $M^{\circ}$ as
\begin{equation*}
    G\times M^{\circ} \rightarrow M^{\circ};\,(g,f)\mapsto [x\mapsto f(g^{-1}x)]. 
\end{equation*}
\item Let $M$ be a $G$-lattice. For a normal subgroup $N$ of $G$, we define $G/N$-lattices $M^{N}$ and $M^{[N]}$ as follows: 
\begin{equation*}
M^{N}:=\{x\in M\mid n(x)=x\text{ for all }n\in N\},\quad 
M^{[N]}:=\left((M^{\circ})^{N}\right)^{\circ}. 
\end{equation*}
Note that $M^{[N]}$ is isomorphic to $M_{N}/M_{N,\tor}$, where $M_{N}$ denotes the $N$-coinvariant part of $M$, and $M_{N,\tor}$ is the torsion part of $M_{N}$. 
\end{itemize}
\end{nota}

\section{Basic facts on the rationality of tori}\label{sect:rvrt}

Let $k$ be a field. For a non-negative integer $n$, we denote by $\P_{k}^{n}$ the projective space of dimension $n$ over $k$. Consider an algebraic variety $X$ over $k$. We say that $X$ is
\begin{itemize}
    \item \emph{rational} over $k$ if it is birationally equivalent to a projective space over $k$; 
    \item \emph{stably rational} over $k$ if $X\times_{k}\P_{k}^{m}$ is rational over $k$ for some $m\in \Znn$; 
    \item \emph{retract rational} over $k$ if there exist rational maps $f\colon \P_{k}^{n}\dashrightarrow X$ and $g\colon X\dashrightarrow \P_{k}^{n}$ with $n\in \Znn$ such that $f\circ g=\id_{X}$; 
    \item \emph{unirational} over $k$ if there is a dominant rational map from $\P_{k}^{n}$ to $X$ for some $n\in \Znn$. 
\end{itemize}

The notion of retract rationality was originally introduced by Saltman (\cite{Saltman1984}) in the case where $k$ is infinite (see also \cite{Kang2012}). It has been generalized for all varieties over arbitrary fields by Merkurjev (\cite{Merkurjev2017}). Note that one has implications
\begin{center}
rational $\Rightarrow$ stably rational 
$\Rightarrow$ retract rational $\Rightarrow$ unirational. 
\end{center}

In this paper, we concentrate on the case where $X$ is an algebraic torus over $k$. Fix a separable closure $k^{\sep}$ of $k$. Then, we can rephrase the stable rationality and the retract rationality of algebraic tori by means of $G$-lattices, where $G$ is a finite quotient of the Galois group of $k^{\sep}/k$. We follow the same terminology as \cite{Lorenz2005}, \cite{Endo2011} and \cite[Section 2]{Hasegawa}. 

\begin{dfn}[{\cite[\S 1]{Endo2011}, \cite[Definition 2.1]{Hasegawa}}]
Let $G$ be a finite group. 
\begin{enumerate}
\item A $G$-lattice $M$ is said to be \emph{permutation} if it has a $\Z$-basis permuted by $G$, i.e.~
    \begin{equation*}
        M\cong \bigoplus_{i=1}^{m} \Z[G/H_i]
    \end{equation*}
for some subgroups $H_1,\ldots, H_m$ of $G$. 
\item A $G$-lattice $M$ is said to be \emph{stably permutation} if $M\oplus R\cong R'$ for some permutation $G$-lattices $R$ and $R'$.
\item A $G$-lattice $M$ is said to be \emph{quasi-permutation} if there is an exact sequence of $G$-lattices
\begin{equation*}
     0\rightarrow M \rightarrow R \rightarrow F \rightarrow 0, 
\end{equation*}
where $R$ and $F$ are permutation. 
\item A $G$-lattice $M$ is said to be \emph{quasi-invertible} if it is a direct summand of a quasi-permutation $G$-lattice.
\end{enumerate}
\end{dfn}

It is not difficult to confirm the following:
\begin{center}
permutation $\Rightarrow$ stably permutation $\Rightarrow$ quasi-permutation $\Rightarrow$ 
quasi-invertible. 
\end{center}

\begin{dfn}[{\cite[\S\S 2.3--2.5]{Lorenz2005}}]
Let $G$ be a finite group. We say that a $G$-lattice $M$ is
\begin{enumerate}
\item \emph{invertible} (or, \emph{permutation projective}) if it is a direct summand of a permutation $G$-lattice;
\item \emph{coflabby} if $H^{1}(H,M)=0$ for any subgroup $H$ of $G$;
\item \emph{flabby} if $M^{\circ}$ is coflabby. 
\end{enumerate}
\end{dfn}

It is known that the following hold: 
\begin{center}
    stably permutation $\Rightarrow$ invertible $\Rightarrow$ flabby and coflabby. 
\end{center}
Here the rightmost implication is a consequence of \cite[(1.2) Proposition]{Lenstra1974}. 

\vspace{6pt}
Let $G$ be a finite group. We say that $G$-lattices $M_{1}$ and $M_{2}$ are \emph{similar} if there exist permutation $G$-lattices $R_{1}$ and $R_{2}$ such that $M_{1}\oplus R_{1}\cong M_{2}\oplus R_{2}$. We denote by $\sS(G)$ the set of similarity classes of $G$-lattices. For a $G$-lattice $M$, we write $[M]$ for the similarity class containing $M$. Then $\sS(G)$ is a commutative monoid with respect to the sum
\begin{equation*}
    [M_{1}]+[M_{2}]:=[M_{1}\oplus M_{2}]. 
\end{equation*}
By definition, one has the following for any $G$-lattice $M$: 
\begin{itemize}
    \item $[M]=0$ if and only if $M$ is stably permutation; and
    \item $[M]$ is invertible in $\sS(G)$ if and only if $M$ is invertible. 
\end{itemize}

\begin{dfn}[{\cite[\S 2.6]{Lorenz2005}}]
Let $G$ be a finite group, and $M$ a $G$-lattice. 
\begin{enumerate}
\item A \emph{coflabby resolution} of $M$ is an exact sequence of $G$-lattices
\begin{equation*}
0\rightarrow F \rightarrow R \rightarrow M \rightarrow 0, 
\end{equation*}
where $R$ is permutation and $F$ is coflabby. 
\item A \emph{flabby resolution} of $M$ is an exact sequence of $G$-lattices
\begin{equation*}
0\rightarrow M \rightarrow R \rightarrow F \rightarrow 0, 
\end{equation*}
where $R$ is permutation and $F$ is flabby. 
\end{enumerate}
\end{dfn}

There is a coflabby resolution for any $G$-lattice, which is a consequence of \cite[Lemma 1.1]{Endo1975}. This implies the existence of a flabby resolution of every $G$-lattice. Moreover, if 
\begin{equation*}
0\rightarrow M \rightarrow R \rightarrow F \rightarrow 0
\end{equation*}
is a flabby resolution of $M$, then the class $[F]$ in $\sS(G)$ only depends on $M$. In the sequel, we denote $[F]$ by $[M]^{\fl}$. It is known that the map
\begin{equation*}
    \sS(G)\rightarrow \sS(G);\,[M]\mapsto [M]^{\fl}
\end{equation*}
is an endomorphism of monoids. In particular, we obtain implications as follows: 
\begin{center}
    stably permutation $\Rightarrow$ quasi-permutation,\quad invertible $\Rightarrow$ quasi-invertible. 
\end{center}

\begin{prop}[{\cite[Proposition 2.6]{Hasegawa}}]\label{prop:rtrd}
Let $G$ be a finite group, and $H$ its subgroup. Consider a $G$-lattice $M$. If $M$ is a quasi-permutation (resp.~quasi-invertible) $G$-module, then it is so as an $H$-lattice. 
\end{prop}

\begin{lem}[{\cite[p.~179, Lemme 2 (i), (ii), (iii)]{ColliotThelene1977}}]\label{lem:qtfx}
Let $G$ be a finite group, and $N$ its normal subgroup. Consider a $G$-lattice $M$. 
\begin{enumerate}
\item If $M$ is a permutation $G$-lattice, then $M^{N}$ is a permutation $G/N$-lattice. 
\item If $M$ is a coflabby $G$-lattice, then $M^{N}$ is a coflabby $G/N$-lattice. 
\item Let
\begin{equation*}
0\rightarrow F \rightarrow R \rightarrow M \rightarrow 0
\end{equation*}
be a coflabby resolution of $M$ in $G$-lattices. Then
\begin{equation*}
0\rightarrow F^{N} \rightarrow R^{N} \rightarrow M^{N} \rightarrow 0
\end{equation*}
is a coflabby resolution of $M^{N}$ in $G/N$-lattices. 
\end{enumerate}
\end{lem}

For a torus $T$ over a field $k$, we define the cocharacter module $X_{*}(T)$ and the character module $X^{*}(T)$ as follows: 
\begin{equation*}
X_{*}(T):=\Hom(\G_{m,k^{\sep}},T\otimes_{k}k^{\sep}),\quad 
X^{*}(T):=\Hom(T\otimes_{k}k^{\sep},\G_{m,k^{\sep}}). 
\end{equation*}
These are finite free abelian groups equipped with continuous actions of $\Gal(k^{\sep}/k)$ (with respect to discrete topology).

\begin{prop}[{\cite[Proposition 9.5.3, Proposition 9.5.4]{Lorenz2005}}]\label{prop:rtiv}
Let $k$ be a field, and $T$ a torus over $k$ which splits over a finite Galois extension $L$ of $k$. Put $G:=\Gal(L/k)$. Then $T$ is stably rational (resp.~retract rational) over $k$ if and only if the $G$-lattice $X^{*}(T)$ is quasi-permutation (resp.~quasi-invertible). 
\end{prop}

\section{Multinorm one tori and their character groups}\label{sect:mnot}

Here, we recall the theory on $I_{G/\cH}^{(\varphi)}$ and $J_{G/\cH}^{(\varphi)}$ in \cite[\S 3]{Hasegawa}. Let $H$ be a subgroup of $G$. Then, one has a surjection
\begin{equation*}
    \varepsilon_{G/H}\colon \Z[G/H]\rightarrow \Z;\,\sum_{g\in G/H}a_{g}g\mapsto \sum_{g\in G/H}a_{g}. 
\end{equation*}
Let $\cH$ be a multiset  of subgroups of $G$, which is assumed to be finite and non-empty in the sequel. We define a $G$-lattice $I_{G/\cH}$ by an exact sequence
\begin{equation*}
    0\rightarrow I_{G/\cH}\rightarrow \bigoplus_{H\in \cH}\Z[G/H]\xrightarrow{(\varepsilon_{G/H})_{H\in \cH}} \Z \rightarrow 0. 
\end{equation*}
Furthermore, we define $J_{G/\cH}:=I_{G/\cH}^{\circ}$. 

\begin{prop}[{\cite[Proposition 3.1]{Hasegawa}}]\label{prop:coch}
Let $k$ be a field, $\bK=\prod_{i=1}^{r}K_{i}$ an {\'e}tale algebra over $k$. Take a finite Galois extension $L$ of $k$ containing $K_{1},\ldots,K_{r}$. Put $G:=\Gal(L/k)$ and $\cH:=\{\Gal(L/K_{i})\mid i\in \{1,\ldots,r\}\}$. Then there exist isomorphisms of $G$-lattices
\begin{equation*}
    X_{*}(T_{\bK/k})\cong I_{G/\cH},\quad X^{*}(T_{\bK/k})\cong J_{G/\cH}. 
\end{equation*}
In particular, $T_{\bK/k}$ is stably (resp.~retract) rational over $k$ if and only if the $G$-lattice $J_{G/\cH}$ is quasi-permutation (resp.~quasi-invertible). 
\end{prop}

In this paper, we mainly determine whether the $G$-lattice $J_{G/\cH}$ is quasi-permutation or quasi-invertible in order to assess the rationality problem for multinorm one tori.

For a multiset $\cH$ of subgroups of $G$, we denote by $N^{G}(\cH)$ the maximal normal subgroup of $G$ which is contained in $H$ for all $H\in \cH$. Moreover, we simply denote $N^{G}(\cH)$ by $N^{G}(H)$ if $\cH$ consists of a single subgroup $H$. 

According to \cite{Hasegawa}, we use the following notation for a multiset $\cH$ of subgroups in a finite group $G$. 
\begin{itemize}
\item Denote by $\cH^{\set}$ the underlying set of $\cH$. 
\item For $H\in \cH^{\set}$, write $m_{\cH}(H)$ for the multiplicity of $H$ in $\cH$. 
\end{itemize}
We define $\Delta$ as follows:
\begin{equation*}
    \Delta:=\coprod_{m \in \Zpn} \Delta_{m}. 
\end{equation*}
Here, $\Delta_{m}:=\{(d_{1},\ldots,d_{m})\in (\Znn)^{m}\mid d_{1}\leq \cdots \leq d_{m}\}$ for each $m\in \Zpn$. 
For $\bd \in \Delta_{m}$ and $i\in \{1,\ldots,m\}$, we denote by $\bd_{i}$ the $i$-th factor of $\bd$. 

\begin{dfn}[{\cite[Definition 3.2]{Hasegawa}}]
Let $\cH$ be a multiset. A \emph{weight function on $\cH$} is defined as a map
\begin{equation*}
\varphi \colon \cH^{\set}\rightarrow \Delta
\end{equation*}
such that $\varphi(H)\in \Delta_{m_{\cH}(H)}$ for any $H\in \cH^{\set}$. 
\end{dfn}

Let $a_{1},\ldots,a_{m}\in \Z$. If $a_{i}\neq 0$ for some $i$, we denote by $\gcd(a_{1},\ldots,a_{m})\in \Zpn$ the usual one, that is, the greatest common divisor of $a_{1},\ldots,a_{m}$. On the other hand, set
\begin{equation*}
\gcd(0,\ldots,0):=0. 
\end{equation*}

\begin{dfn}[{\cite[Definition 3.3]{Hasegawa}}]\label{dfn:dfdp}
Let $G$ be a finite group, $\cH$ a multiset of its subgroups, and $\varphi$ a weight function on $\cH$. 
\begin{enumerate}
\item We define a weight function $d_{\varphi}$ of $\cH^{\set}$ as follows: 
\begin{equation*}
d_{\varphi}\colon \cH^{\set} \rightarrow \Delta_{1};\,H \mapsto 
\gcd(\varphi(H)_{1},\ldots,\varphi(H)_{m_{\cH}(H)}). 
\end{equation*}
\item Put $d(\varphi):=\gcd(d_{\varphi}(H)\mid H\in \cH^{\set})$. We say that $\varphi$ is \emph{normalized} if $d(\varphi)=1$. 
\end{enumerate}
\end{dfn}

\begin{lem}[{\cite[Lemma 3.4]{Hasegawa}}]\label{lem:pnor}
Let $G$ be a finite group, $\cH$ a multiset of its subgroups, and $\varphi$ a weight function on $\cH$ satisfying $d:=d(\varphi)\neq 0$. Then the map
\begin{equation*}
\varphi^{\nor} \colon \cH^{\set} \rightarrow \Delta;\,H \mapsto (d^{-1}\varphi(H)_{1},\ldots,d^{-1}\varphi(H)_{m_{\cH}(H)})
\end{equation*}
is a normalized weight function of $\cH$. 
\end{lem}

\begin{dfn}[{\cite[Definition 3.5]{Hasegawa}}]\label{dfn:gnmn}
Let $G$ be a finite group, and $\cH$ a multiset of its subgroups. Consider a weight function $\varphi$ of $\cH$. We define a $G$-lattice $I_{G/\cH}^{(\varphi)}$ by the exact sequence
\begin{equation}\label{eq:ipdf}
0\rightarrow I_{G/\cH}^{(\varphi)}\rightarrow \bigoplus_{H\in \cH^{\set}}\Z[G/H]^{\oplus m_{\cH}(H)}\xrightarrow{(\varphi(H)\cdot \varepsilon_{G/H})_{H}} \Z. 
\end{equation}
Here, the map $\varphi(H)\cdot \varepsilon_{G/H}\colon \Z[G/H]^{\oplus m_{\cH}(H)}\rightarrow \Z$ is defined as the sum of $\varphi(H)_{i}\cdot \varepsilon_{G/H}$ for all $i\in \{1,\ldots,m_{\cH}(H)\}$. Furthermore, set 
\begin{equation*}
J_{G/\cH}^{(\varphi)}:=(I_{G/\cH}^{(\varphi)})^{\circ}. 
\end{equation*}
\end{dfn}

If $\varphi$ is normalized, then the rightmost homomorphism of \eqref{eq:ipdf} is surjective. Moreover, we have an exact sequence of $G$-lattices
\begin{equation*}
    0 \rightarrow \Z \xrightarrow{\varphi(H)\cdot \varepsilon_{G/H}^{\circ}}
    \bigoplus_{H\in \cH^{\set}}\Z[G/H]^{\oplus m_{\cH}(H)}\rightarrow 
    J_{G/\cH}^{(\varphi)}\rightarrow 0. 
\end{equation*}

The following summarizes the results that will be needed in this paper.

\begin{dfn}[{\cite[Definition 3.13]{Hasegawa}}]\label{dfn:reds}
Let $G$ be a finite group, and $\cH$ a set of its subgroups (that is, all elements in $\cH$ have multiplicity $1$). We say that $\cH$ is \emph{reduced} if $H\not\subset H'$ for any $H,H'\in \cH$ with $H\neq H'$. 
\end{dfn}

For a multiset $\cH$ of subgroups of a finite group $G$, we denote by $\cH^{\red}$ the subset of $\cH^{\set}$ consisting of all elements of $\cH^{\set}$ that are maximal with respect to inclusion. Note that it is reduced in the sense of Definition \ref{dfn:reds}. 

\begin{prop}[{\cite[Corollary 3.14]{Hasegawa}}]\label{prop:rdrd}
Let $G$ be a finite group, and $\cH$ a multiset of its subgroups. 
\begin{enumerate}
\item Let $\varphi$ be a normalized weight function on $\cH$, and $H'_{0}\in \cH^{\set}$. Then the following hold in $\sS(G)$: 
\begin{equation*}
    [I_{G/\cH}^{(\varphi)}]=[I_{G/\cH^{\set}}^{(d_{\varphi})}],\quad [J_{G/\cH}^{(\varphi)}]=[J_{G/\cH^{\set}}^{(d_{\varphi})}]. 
\end{equation*}
\item Then, the following hold in $\sS(G)$: 
\begin{equation*}
    [I_{G/\cH}]=[I_{G/\cH^{\red}}],\quad [J_{G/\cH}]=[J_{G/\cH^{\red}}]. 
\end{equation*}
In particular, the $G$-lattice $J_{G/\cH}$ is quasi-permutation (resp.~quasi-invertible) if and only if $J_{G/\cH^{\red}}$ is so. 
\end{enumerate}
\end{prop}

\begin{prop}[{\cite[Lemma 3.16]{Hasegawa}}]\label{prop:conj}
Let $G$ be a finite group, and $\cH$ a multiset of its subgroups. Take $H_{0}\in \cH$ and $g\in G\setminus N_{G}(H_{0})$. We define a multiset $\cH'$ of subgroups of $G$ as follows: 
\begin{equation*}
m_{\cH'}(H')=
\begin{cases}
m_{\cH}(gH_{0}g^{-1})+1&\text{if }H'=gH_{0}g^{-1}; \\
m_{\cH}(H_{0})-1&\text{if }H'=H_{0};\\
m_{\cH}(H')&\text{otherwise}. 
\end{cases}
\end{equation*}
Then there exist isomorphisms of $G$-lattices
\begin{equation*}
    I_{G/\cH}\cong I_{G/\cH'},\quad J_{G/\cH}\cong J_{G/\cH'}. 
\end{equation*}
\end{prop}

\begin{dfn}[{\cite[Definition 3.17]{Hasegawa}}]
Let $G$ be a finite group. We say that a set of subgroups $\cH$ of $G$ is \emph{strongly reduced} if $H \not\subset gH'g^{-1}$ for any $H,H'\in \cH$ with $H\neq H'$ as subgroups of $G$ and any $g\in G$. 
\end{dfn}

For a multiset $\cH$ of subgroups of a finite group $G$, let $\cH^{\srd}$ be a maximal strongly reduced subset of the set of all elements of $\cH^{\set}$ that are maximal with respect to inclusion up to $G$-conjugacy. 
Note that $\cH^{\srd}$ is not uniquely determined. 
However, we have $\cH^{\srd}\subset \cH^{\red}$ by definition. 

\begin{prop}[{\cite[Proposition 3.18]{Hasegawa}}]\label{prop:rdsr}
Let $G$ be a finite group, and $\cH$ a multiset of its subgroups. Then, the following hold in $\sS(G)$: 
\begin{equation*}
    [I_{G/\cH}]=[I_{G/\cH^{\srd}}],\quad [J_{G/\cH}]=[J_{G/\cH^{\srd}}]. 
\end{equation*}
In particular, the $G$-lattice $J_{G/\cH}$ is quasi-permutation (resp.~quasi-invertible) if and only if $J_{G/\cH^{\srd}}$ is so. 
\end{prop}

\begin{prop}[{\cite[Proposition 3.19]{Hasegawa}}]\label{prop:rdtr}
Let $G$ be a finite group, $\cH$ a multiset of its subgroups, and $\varphi$ a normalized weight function on $\cH$. Assume that there exist $H_{0},H'_{0}\in \cH$ and $i_{0}\in \{1,\ldots,m_{\cH}(H_{0})\}$ such that
\begin{itemize}
\item $H_{0}'\subsetneq H_{0}$; and
\item $\varphi(H_{0})_{i_{0}}\in (H_{0}:H'_{0})d_{\varphi}(H'_{0})\Z$. 
\end{itemize}
We define a multiset $\cH'$ of subgroups of $G$ as follows: 
\begin{equation*}
m_{\cH'}(H)=
\begin{cases}
    m_{\cH}(H_{0})-1&\text{if }H=H_{0};\\
    m_{\cH}(H)&\text{if }H\neq H_{0}.
\end{cases}
\end{equation*}
Furthermore, denote by $\varphi'$ the weight function of $\cH'$ defined as 
\begin{equation*}
    \varphi'(H'):=
    \begin{cases}
        (\varphi(H_{0})_{1},\ldots,\varphi(H_{0})_{i_{0}-1},\varphi(H_{0})_{i_{0}+1},\ldots,\varphi(H_{0})_{m_{\cH}(H_{0})})&\text{if }H'=H_{0}\in (\cH')^{\set};\\
        \varphi(H')&\text{if }H'\neq H_{0},
    \end{cases}
\end{equation*}
which is normalized.
Then there exist isomorphisms of $G$-lattices
\begin{equation}\label{eq:incd}
I_{G/\cH}^{(\varphi)} \cong I_{G/\cH'}^{(\varphi')}\oplus \Z[G/H_{0}],\quad 
J_{G/\cH}^{(\varphi)} \cong J_{G/\cH'}^{(\varphi')}\oplus \Z[G/H_{0}]. 
\end{equation}
In particular, the $G$-lattice $J_{G/\cH}^{(\varphi)}$ is quasi-permutation (resp.~quasi-invertible) if and only if $J_{G/\cH'}^{(\varphi')}$ is so. 
\end{prop}

\begin{prop}[{\cite[Proposition 3.22]{Hasegawa}}]\label{prop:rtmc}
Let $G$ be a finite group, and $P$ a subgroup of $G$. Consider a multiset $\cH$ of subgroups of $G$ and a normalized weight function $\varphi$ on $\cH$. Then there are isomorphisms of $P$-modules
\begin{equation*}
    I_{G/\cH}^{(\varphi)}\cong I_{P/\cH_{P}}^{(\varphi_{P})},\quad J_{G/\cH}^{(\varphi)}\cong J_{P/\cH_{P}}^{(\varphi_{P})}. 
\end{equation*}
Here, $\varphi_{P}$ is defined as follows:
\begin{itemize}
    \item $C_{H}$ is a complete set of representatives of $P\backslash G/H$ in $G$ for each $H \in \cH$; 
    \item $\cH_{P}$ is the multiset of subgroups of $P$ defined as
    \begin{equation*}
        \cH_{P}^{\set}:=\{P\cap gHg^{-1}<P\mid H\in \cH^{\set},g\in C_{H}\}
    \end{equation*}
    and 
    \begin{equation*}
        m_{\cH_{P}}(H'):=\sum_{H\in \cH^{\set},g\in C_{H},P\cap gHg^{-1}=H'}m_{\cH}(H)
    \end{equation*}
    for each $H'\in \cH_{P}^{\set}$; 
    \item $\varphi_{P}$ is the normalized weight function of $\cH_{P}$ which sends $H'\in \cH_{P}^{\set}$ to the unique element of $\Delta_{m_{\cH_{P}}(H')}$ determined by $\varphi(H)$ for all $H\in \cH$ and $g\in C_{H}$ with $H'=P\cap gHg^{-1}$. 
\end{itemize}
In particular, if $J_{G/\cH}^{(\varphi)}$ is quasi-permutation (resp.~quasi-invertible), then $J_{P/\cH_{P}}^{(\varphi_{P})}$ is so. 
\end{prop}

\begin{prop}[{\cite[Proposition 3.24]{Hasegawa}}]\label{prop:tkfx}
Let $G$ be a finite group, and $N$ its normal subgroup. Take a multiset $\cH$ of subgroups of $G$ and a normalized weight function $\varphi$ of $\cH$. Then there are isomorphisms
\begin{equation*}
(I_{G/\cH}^{(\varphi)})^{N}\cong I_{(G/N)/\cH^{N}}^{(\overline{\varphi}_{G/N}^{\nor})}, \quad 
(J_{G/\cH}^{(\varphi)})^{[N]}\cong J_{(G/N)/\cH^{N}}^{(\overline{\varphi}_{G/N}^{\nor})}. 
\end{equation*}
Here $\cH^{N}$ and $\overline{\varphi}_{G/N}^{\nor}$ are defined as follows: 
\begin{itemize}
\item $\cH^{N}$ is the multiset of subgroups of $G/N$ defined by
\begin{equation*}
    (\cH^{N})^{\set}=\{HN/N<G/N\mid H\in \cH^{\set}\}
\end{equation*}
and
\begin{equation*}
    m_{\cH^{N}}(\overline{H})=\sum_{H\in \cH^{\set},HN/N=\overline{H}}m_{\cH}(H)
\end{equation*}
for each $\overline{H}\in (\cH^{N})^{\set}$; 
\item the weight function $\overline{\varphi}_{G/N}$ on $\cH^{N}$ that sends $\overline{H}\in (\cH^{N})^{\set}$ to the unique element of $\Delta_{m_{\cH^{N}}}(\overline{H})$ determined by $(HN:H)\varphi(H)$ for all $H\in \cH$ with $HN/N=\overline{H}$; and
\item $\overline{\varphi}_{G/N}^{\nor}$ is as in Lemma \ref{lem:pnor}. 
\end{itemize}
In particular, if $J_{G/\cH}^{(\varphi)}$ is quasi-permutation (resp.~quasi-invertible), then $J_{(G/N)/\cH^{N}}^{(\overline{\varphi}_{G/N}^{\nor})}$ is so.  
\end{prop}

\begin{cor}[{\cite[Corollary 3.25]{Hasegawa}}]\label{cor:trfx}
Let $G$ be a finite group, and $\cH$ a multiset of subgroups of $G$. Consider a normal subgroup $N$ of $G$ that is contained in $N^{G}(\cH)$. Then there exist isomorphisms of $G$-lattices
\begin{equation*}
    I_{G/\cH}\cong I_{(G/N)/\cH^{N}},\quad J_{G/\cH}\cong J_{(G/N)/\cH^{N}}. 
\end{equation*}
Here we regard $I_{(G/N)/\cH^{N}}$ and $J_{(G/N)/\cH^{N}}$ as $G$-lattices by the natural surjection $G\twoheadrightarrow G/N$. In particular, the $G$-lattice $J_{G/\cH}$ is quasi-permutation (resp.~quasi-invertible) if and only if $J_{(G/N)/\cH^{N}}$ is so as a $G/N$-lattice. 
\end{cor}

\section{Proof of Theorems \ref{mtd1} and \ref{mth3}}\label{sect:pf13}

\begin{dfn}
Let $G$ be a finite group, and $\cH$ a multiset of its subgroups. Then we define a positive integer $d_{G}(\cH)$ as follows: 
\begin{equation*}
    d_{G}(\cH):=\gcd((G:H)\mid H\in \cH). 
\end{equation*}
\end{dfn}

\begin{thm}\label{thm:bigh}
Let $G$ be a finite group, and $\cH$ a multiset of its subgroups of $G$. We further assume that there is a subgroup $H_{0}$ of $G$ such that $(G:H_{0})=d_{G}(\cH)$ and $H\subset H_{0}$ for any $H\in \cH^{\set}$. Then, we have the following in $\sS(G)$: 
\begin{equation*}
    [I_{G/\cH}]=[I_{G/H_{0}}],\quad [J_{G/\cH}]=[J_{G/H_{0}}]. 
\end{equation*}
\end{thm}

\begin{proof}
It suffices to prove $[I_{G/\cH}]=[I_{G/H_{0}}]$. 

\textbf{Case 1.~$H_{0}\in \cH^{\set}$. }

In this case, the assertion follows from Proposition \ref{prop:rdrd}. 

\textbf{Case 2.~$H_{0}\notin \cH^{\set}$. }

In this case, we have $\#\cH^{\set}\geq 2$. Write $\cH^{\set}=\{H_{1},\ldots,H_{r}\}$ with $r\geq 2$ so that 
\begin{equation*}
    (G:H_{1})\leq \cdots \leq (G:H_{r}). 
\end{equation*}
Define a multiset $\widetilde{\cH}$ of subgroups of $G$ as $\widetilde{\cH}^{\set}=\{H_{0},H_{1},\ldots,H_{r}\}$ and
\begin{equation*}
    m_{\widetilde{\cH}}(H_{i})=
    \begin{cases}
        m_{\cH}(H_{i}) &\text{if }i\geq 1;\\
        \#\cH^{\set} &\text{if }i=0. 
    \end{cases}
\end{equation*}
Furthermore, the weight function $\varphi$ of $\widetilde{\cH}$ is defined as follows: 
\begin{equation*}
    \varphi \colon \widetilde{\cH}^{\set}\rightarrow \Delta;\,H\mapsto 
    \begin{cases}
        (\underbrace{1,\ldots,1}_{m_{\cH}(H)})&\text{if }H\in \cH^{\set};\\
        ((H_{0}:H_{1}),\ldots,(H_{0}:H_{r}))&\text{if }H=H_{0}. 
    \end{cases}
\end{equation*}
Then, we have $[I_{G/\widetilde{\cH}}^{(\varphi)}]=[I_{G/\cH}]$ by Proposition \ref{prop:rdtr}. On the other hand, Proposition \ref{prop:rdrd} (i) gives an equality $[I_{G/\widetilde{\cH}}^{(\varphi)}]=[I_{G/\widetilde{\cH}^{\set}}^{(d_{\varphi})}]$. Now, since 
\begin{equation*}
    \gcd((H_{0}:H_{1}),\ldots,(H_{0}:H_{r}))=\frac{d_{G}(\cH)}{(G:H_{0})}=1, 
\end{equation*}
the weight function $d_{\varphi}$ of $\widetilde{\cH}^{\set}$ takes the constant value $1$. In particular, we have $[I_{G/\widetilde{\cH}^{\set}}^{(d_{\varphi})}]=[I_{G/\widetilde{\cH}^{\set}}]$. Consequently, we get an equality in $\sS(G)$ as follows: 
\begin{equation*}
    [I_{G/\cH}]=[I_{G/\widetilde{\cH}^{\set}}]. 
\end{equation*}
Now, we have $H_{0}\in \widetilde{\cH}^{\set}$ and $d_{G}(\widetilde{\cH}^{\set})=(G:H_{0})$, and hence the assertion follows from Case 1. 
\end{proof}

\begin{cor}\label{cor:gcdo}
Let $G$ be a finite group, and $\cH$ a multiset of its subgroups with $d_{G}(\cH)=1$. Then the $G$-lattices $I_{G/\cH}$ and $J_{G/\cH}$ are stably permutation. In particular, $J_{G/\cH}$ is a quasi-permutation $G$-lattice. 
\end{cor}

\begin{proof}
This follows from Theorem \ref{thm:bigh} since $H_{0}:=G$ satisfies the assumptions of Theorem \ref{thm:bigh}. 
\end{proof}

\begin{proof}[Proof of Theorem \ref{mtd1}]
Let $L/k$ be a finite Galois field extension containing $K_{1},\ldots,K_{r}$. Put $G:=\Gal(L/k)$ and $\cH:=\{\Gal(L/K_{i})\mid i\in \{1,\ldots,r\}\}$. By Proposition \ref{prop:coch}, it suffices to prove that the $G$-lattice $J_{G/\cH}$ is quasi-permutation. However, the assertion follows from Corollary \ref{cor:gcdo}. 
\end{proof}

For a prime number $p$, let $\ord_{p}\colon \Q^{\times}\rightarrow \Z$ be the homomorphism defined as
\begin{equation*}
\ord_{p}(\ell)=
\begin{cases}
1&\text{if }\ell=p;\\
0&\text{if }\ell \neq p
\end{cases}
\end{equation*}
for each prime number $\ell$. 

\begin{lem}\label{lem:psyl}
Let $G$ be a finite group, and $H$ its subgroup. Take a prime divisor $p$ of $\#G$, and pick a $p$-Sylow subgroup $P$ of $G$. Then we have
    \begin{equation*}
        \ord_{p}(P:P\cap gHg^{-1})\geq \ord_{p}(G:H)
    \end{equation*}
    for any $g\in G$. Moreover, the equality holds for some $g\in G$. 
\end{lem}

\begin{proof}
The former assertion follows from
\begin{equation*}
    \ord_{p}(\#(P\cap gHg^{-1}))\leq \ord_{p}(\#gHg^{-1})=\ord_{p}(\#H)
\end{equation*}
for every $g\in G$. The latter is a consequence of Sylow's theorem. 
\end{proof}

Let $G$ be a finite group, and $\cH$ a multiset of its subgroups. For each prime divisor $p$ of $\#G$, we define $\mu_{p}(\cH)$ and $M_{p}(\cH)$ as follows: 
\begin{equation*}
\mu_{p}(\cH):=\min\{\ord_{p}(G:H)\in \Znn \mid H\in \cH\},\quad M_{p}(\cH):=\max \{\ord_{p}(G:H)\in \Znn \mid H\in \cH\}. 
\end{equation*}

\begin{cor}\label{cor:hppo}
Let $G$ be a finite group, and $\cH$ a multiset of its subgroups. 
Take a prime divisor $p$ of $\#G$, and choose a $p$-Sylow subgroup $P$ of $G$. 
\begin{enumerate}
    \item We have $\ord_{p}d_{G}(\cH)=\mu_{p}(\cH_{P})=\mu_{p}(\cH_{P}^{\red})=\mu_{p}(\cH_{P}^{\srd})$. 
    \item For any subgroup $D$ of $G$ containing $P$, we have $\ord_{p}d_{G}(\cH)=\ord_{p}d_{D}(\cH_{D})$. 
\end{enumerate}
\end{cor}

\begin{proof}
(i): This follows from Lemma \ref{lem:psyl} since $\ord_{p}(d_{G}(\cH))=\min\{\ord_{p}(G:H)\mid H\in \cH\}$. 

(ii): This is a consequence of (i).  
\end{proof}

\begin{thm}\label{thm:qips}
Let $G$ be a finite group, and $\cH$ a multiset of its subgroups. For each prime divisor $p$ of $d_{G}(\cH)$, pick a $p$-Sylow subgroup $G_{p}$ of $G$. Then the following are equivalent: 
\begin{enumerate}
\item $J_{G/\cH}$ is a quasi-invertible $G$-lattice; 
\item $J_{G/\cH}$ is a quasi-invertible $G_{p}$-lattice for any prime divisor $p$ of $d_{G}(\cH)$; 
\item $J_{G_{p}/\cH_{G_{p}}}$ is a quasi-invertible $G_{p}$-lattice for any prime divisor $p$ of $d_{G}(\cH)$; 
\item $J_{G_{p}/\cH_{G_{p}}^{\red}}$ is a quasi-invertible $G_{p}$-lattice for any prime divisor $p$ of $d_{G}(\cH)$; 
\item $J_{G_{p}/\cH_{G_{p}}^{\srd}}$ is a quasi-invertible $G_{p}$-lattice for any prime divisor $p$ of $d_{G}(\cH)$. 
\end{enumerate}
\end{thm}

\begin{proof}
(i) $\Rightarrow$ (ii) is clear. (ii) $\Leftrightarrow$ (iii) is a consequence of Proposition \ref{prop:rtmc}. Furthermore, (iii) $\Leftrightarrow$ (iv) and (iii) $\Leftrightarrow$ (v) are consequences of Propositions \ref{prop:rdrd} and \ref{prop:rdsr}, respectively. Hence, it suffices to prove (i) assuming (ii)--(v). 

Consider a prime divisor $p$ of $\#G$ that does not divide $d_{G}(\cH)$, and pick a $p$-Sylow subgroup $G_{p}$ of $G$. Then, we have $\mu_{p}(\cH_{G_{p}})=0$ by Corollary \ref{cor:hppo} (i). Moreover, applying Corollary \ref{cor:hppo} (ii), we obtain $\ord_{p}d_{G_{p}}(\cH_{G_{p}})=0$. This implies $d_{G_{p}}(\cH_{G_{p}})=1$ since $G_{p}$ is a $p$-group. Hence, by Corollary \ref{cor:gcdo}, the $G_{p}$-lattice $J_{G_{p}/\cH_{G_{p}}}$ is quasi-permutation. Consequently, $J_{G/\cH}$ is a quasi-permutation $G_{p}$-lattice. Combining the above result with assumption (ii), we obtain that $J_{G/\cH}$ is quasi-invertible as a $G_{p}$-lattice for all prime divisors $p$ of $\#G$. Therefore, (i) follows from  \cite[Corollary 1.6]{Endo2011}. 
\end{proof}

\begin{proof}[Proof of Theorem \ref{mth3}]
Set $G:=\Gal(L/k)$, $H_{i}:=\Gal(L/K_{i})$ for each $i\in \{1,\ldots,r\}$, and $G_{p}:=\Gal(L/L^{(p)})$ for each prime divisor $p$ of $d_{k}(\bK)$. Consider the multiset $\cH:=\{H_{1},\ldots,H_{r}\}$. Then, Proposition \ref{prop:coch} implies that $T_{\bK/k}$ is retract rational over $k$ (resp.~$L^{(p)}$ for any prime divisor $p$ of $d_{k}(\bK)$) if and only if $J_{G/\cH}$ is quasi-invertible as a $G$-lattice (resp.~$G_{p}$-lattice for any prime divisor of $d_{G}(\cH)$). Hence, the assertion follows from Theorem \ref{thm:qips}. 
\end{proof}

Finally, note that whether (iv) and (v) in Theorem \ref{thm:qips} hold can be verified using the following theorems obtained by the authors (\cite{Hasegawa}). Therefore, we acquire a criterion for the $G$-lattice $J_{G/\cH}$ to be quasi-invertible. 

\begin{prop}[{\cite[Theorem 5.10]{Hasegawa}}]\label{prop:odqp}
Let $p$ be an odd prime number, and $G$ a $p$-group. Consider a reduced set $\cH$ of subgroups of $G$. Then the following are equivalent: 
\begin{enumerate}
    \item $J_{G/\cH}$ is a quasi-permutation $G$-lattice; 
    \item $J_{G/\cH}$ is a quasi-invertible $G$-lattice; 
    \item $\#\cH=1$ and $G/N^{G}(\cH)$ is cyclic. 
\end{enumerate}
\end{prop}

\begin{prop}[{\cite[Theorem 7.15]{Hasegawa}}]\label{prop:tgqp}
Let $G$ be a $2$-group, and $\cH$ a strongly reduced set of its subgroups. Then the $G$-lattice $J_{G/\cH}$ is quasi-permutation if and only if it is quasi-invertible. Moreover, the above two conditions hold if and only if 
\begin{enumerate}
\item $\#\cH=1$ and $G/N^{G}(\cH)$ is cyclic; or
\item $\cH=\{H,H'\}$ and there exist $\nu \in \Zpn$ and an isomorphism $G/N^{G}(\cH) \cong D_{2^{\nu}}$ that induces $H/N^{G}(\cH)\cong \langle \sigma_{2^{\nu}}^{2m}\tau_{2^{\nu}} \rangle$ and $H'/{N^{G}(\cH)}\cong \langle \sigma_{2^{\nu}}^{2m'+1}\tau_{2^{\nu}} \rangle$ for some $m,m'\in \Z$. 
\end{enumerate}
\end{prop}

\section{Proof of Theorems \ref{mtd2} and \ref{mtd3}}\label{sect:pf23}

We first prepare the notion on multisets of subgroups in finite groups that corresponds to additional assumptions in Theorems \ref{mtd2} and \ref{mtd3}. 

\begin{dfn}
Let $G$ be a finite group. A set $\cH$ of subgroups is said to be \emph{special} if 
\begin{itemize}
\item $G\notin \cH$; 
\item $\#\cH \geq 2$; and
\item $N^{G}(H)H'=G$ or $N^{G}(H')H=G$ for any $H,H'\in \cH$ with $H\neq H'$. 
\end{itemize}
\end{dfn}

By definition, a special set $\cH$ of subgroups of $G$ satisfies $\#\cH \geq 2$. 

\begin{lem}\label{lem:gnpi}
Let $G$ be a finite group, $H$ a subgroup of $G$, and $N$ a normal subgroup of $G$. Assume $NH=G$ and that there is a prime number $p$ which divides $(G:H)$ and $(G:N)$. Take a $p$-Sylow subgroup $P$ of $G$. Assume that $g\in G$ satisfies $\ord_{p}(P:P\cap gHg^{-1})=\ord_{p}(G:H)$. Then one has an equality
\begin{equation*}
    (P\cap N)\cdot (P\cap gHg^{-1})=P. 
\end{equation*}
\end{lem}

\begin{proof}
By the assumption on $g$, the group $gHg^{-1}$ contains $P\cap gHg^{-1}$ as its $p$-Sylow subgroup. Applying Lemma \ref{lem:psyl} for $gHg^{-1}$ and its normal subgroup $gHg^{-1}\cap N$, we obtain an equality
\begin{equation*}
    \ord_{p}(P\cap gHg^{-1}:P\cap gHg^{-1}\cap N)=\ord_{p}(gHg^{-1}:gHg^{-1}\cap N). 
\end{equation*}
On the other hand, $NH=G$ gives an equality
\begin{equation*}
    (gHg^{-1}:gHg^{-1}\cap N)=(G:N). 
\end{equation*}
Furthermore, one has $\ord_{p}(G:N)=\ord_{p}(P:P\cap N)$ by Lemma \ref{lem:psyl}. Consequently, we obtain an equality
\begin{equation*}
    (P\cap gHg^{-1}:P\cap gHg^{-1}\cap N)=(P:P\cap N), 
\end{equation*}
which gives the desired assertion. 
\end{proof}

\begin{prop}\label{prop:syad}
Let $G$ be a finite group, and $\cH$ a special set of its subgroups with $d_G(\cH)>1$. Take a $p$-Sylow subgroup $P$ of $G$, where $p$ is a prime divisor of $d_{G}(\cH)$. Then there is a special subset $\cH_{P}^{\ast}$ of $\cH_{P}^{\set}$ such that $\#\cH_{P}^{\ast}=\#\cH$, $\mu_{p}(\cH_{P}^{\ast})=\ord_{p}d_{G}(\cH)$ and $M_{p}(\cH_{P}^{\ast})=M_{p}(\cH)$. 
\end{prop}

\begin{proof}
By the assumption on $p$ and Lemma \ref{lem:psyl}, the set $\cH_{P}^{\set}$ does not contain $P$. Write 
\begin{equation*}
\cH=\{H_{1},\ldots,H_{r}\}
\end{equation*}
where $r\geq 2$. For each $i\in \{1,\ldots,r\}$, let $C_{H_{i}}\subset G$ be as in Proposition \ref{prop:rtmc}, and pick $g_{i}\in C_{H_{i}}$ so that
\begin{equation*}
    \ord_{p}(P:P\cap g_{i}H_{i}g_{i}^{-1})=\ord_{p}(G:H_{i}). 
\end{equation*}
Note that such $g_{i}$ exists by Lemma \ref{lem:psyl}. We prove that the subset
\begin{equation*}
    \cH_{P}^{\ast}:=\{P\cap g_{i}H_{i}g_{i}^{-1}\mid i\in \{1,\ldots,r\}\}
\end{equation*}
of $\cH_{P}^{\set}$ is the desired one. Take $i,j\in \{1,\ldots,r\}$ with $i\neq j$. After interchanging $i$ and $j$ if necessary, we may assume $N^{G}(H_{i})H_{j}=G$. Then one has an inclusion
\begin{equation*}
P\cap N^{G}(g_{i}H_{i}g_{i}^{-1})\subset N^{P}(P\cap g_{i}H_{i}g_{i}^{-1}). 
\end{equation*}
Hence, Lemma \ref{lem:gnpi} gives an equality
\begin{equation*}
    N^{P}(P\cap g_{i}H_{i}g_{i}^{-1})\cdot (P\cap g_{j}H_{j}g_{j}^{-1})=P. 
\end{equation*}
On the other hand, the conditions $\#\cH_{P}^{\ast}=\#\cH$, $\mu_{p}(\cH_{P}^{\ast})=\ord_{p}d_{G}(\cH)$ and $M_{p}(\cH_{P}^{\ast})=M_{p}(\cH)$ follow from the construction of $\cH_{P}^{\ast}$. 
\end{proof}

\begin{lem}\label{lem:psrd}
Let $G$ be a finite group, and $\cH$ be a multiset of its subgroups with $G\notin \cH$. 
\begin{enumerate}
\item If $\cH$ is a special set, then it is strongly reduced. 
\item Assume that there is a subset $\cH_{0}$ of $\cH^{\set}$ which is special. Then there is a special subset $\cH_{0}^{\circ}$ of $\cH^{\srd}$ such that $(\cH_{0}\cup \cH_{0}^{\circ})^{\srd}=\cH_{0}^{\circ}$ and $\#\cH_{0}^{\circ}=\#\cH_{0}$. 
\end{enumerate}
\end{lem}

\begin{proof}
(i): This follows from the definition that $\cH$ is special. 

(ii): Write $\cH_{0}=\{H_{1},\ldots,H_{r}\}$, where $r\geq 2$. For each $i\in \{1,\ldots,r\}$, choose $H'_{i}\in \cH^{\srd}$ so that $g_{i}H_{i}g_{i}^{-1}\subset H'_{i}$ for some $g_{i}\in G$. Put $\cH_{0}^{\circ}:=\{H'_{1},\ldots,H'_{r}\}$. Then we have $(\cH_{0}\cup \cH_{0}^{\circ})^{\srd}=\cH_{0}^{\circ}$ by definition. On the other hand, take $i,j\in \{1,\ldots,r\}$ with $i\neq j$. We may assume $N^{G}(H_{i})H_{j}=G$. Then one has $N^{G}(H_{i})g_{j}H_{j}g_{j}^{-1}\subset N^{G}(H'_{i})H'_{j}$. Hence we obtain $N^{G}(H'_{i})H'_{j}=G$. Moreover, $G\notin \cH$ implies $\#\cH_{0}^{\circ}=r$ and $G\notin \cH_{0}^{\circ}$. Therefore, $\cH_{0}^{\circ}$ is the desired one. 
\end{proof}

Now, we can prove Theorems \ref{mtd2} and \ref{mtd3}. 

\begin{thm}\label{thm:gcdt}
Let $G$ be a finite group, and $\cH$ a reduced set of its subgroups that satisfies $\#\cH \geq 2$ and $d_{G}(\cH)=2$. 
\begin{enumerate}
    \item If $\#\cH=2$, we further assume that $\cH$ is special and $(G:H)\in 4\Z$ for some $H\in \cH$. Then the $G$-lattice $J_{G/\cH}$ is not quasi-invertible. 
    \item If $\#\cH \geq 3$, we further assume that there is a special subset $\cH_{0}$ of $\cH$ with $\#\cH_{0}=3$. Then the $G$-lattice $J_{G/\cH}$ is not quasi-invertible. 
\end{enumerate}
\end{thm}


\begin{proof}
Let $P$ be a $2$-Sylow subgroup of $G$. Then Corollary \ref{cor:hppo} (i) gives an equality $\mu_{2}(\cH_{P}^{\srd})=1$. In particular, $(\cH_{P}^{\srd})^{\nor}$ is non-empty. 

(i): Let $\cH_{P}^{\ast}\subset \cH_{P}^{\set}$ be as in Proposition \ref{prop:syad}. Moreover, Lemma \ref{lem:psrd} (ii) gives a special subset $\cH_{P}^{\circ}$ of $\cH_{P}^{\srd}$ such that $(\cH_{P}^{\ast}\cup \cH_{P}^{\circ})^{\srd}=\cH_{P}^{\circ}$ and $\#\cH_{P}^{\circ}=\#\cH_{P}^{\ast}$. Write $\cH=\{H_{1},H_{2}\}$ so that 
\begin{equation*}
1=\ord_{2}(G:H_{1})<\ord_{2}(G:H_{2}). 
\end{equation*}
In the following, we prove $M_{2}(\cH_{P}^{\circ})\geq 2$. Suppose $M_{2}(\cH_{P}^{\circ})=1$. Then we have 
\begin{equation*}
\cH_{P}^{\circ}=\{P\cap g_{1}H_{1}g_{1}^{-1},P\cap g_{2}H_{1}g_{2}^{-1}\}
\end{equation*}
for some $g_{1},g_{2}\in G$. Moreover, $\cH_{P}^{\ast}\cap \cH_{P}^{\circ}$ must be non-empty, which follows from $\mu_{2}(\cH_{P}^{\ast})=1$. We may assume $P \cap g_{1}H_{1}g_{1}^{-1}\in \cH_{P}^{\ast}\cap \cH_{P}^{\circ}$. On the other hand, since $M_{2}(\cH_{P}^{\ast})\geq 2$, there exist $g\in G$ and $h\in P$ such that $\ord_{2}(P:P\cap gH_{2}g^{-1})=\ord_{2}(G:H_{2})$, $P\cap gH_{2}g^{-1}\in \cH_{P}^{\ast}$ and 
\begin{equation}\label{eq:cntr}
    P\cap (hg)H_{2}(hg)^{-1}\subsetneq P\cap g_{2}H_{1}g_{2}^{-1}. 
\end{equation}
We claim that \eqref{eq:cntr} gives a contradiction. If $N^{G}(H_{1})H_{2}=G$, then Lemma \ref{lem:gnpi} implies
\begin{equation*}
(P\cap N^{G}(H_{1}))\cdot (P\cap gH_{2}g^{-1})=P. 
\end{equation*}
Hence, we obtain $N^{P}(P\cap g_{2}H_{1}g_{2}^{-1})\cdot (P\cap (hg)H_{2}(hg)^{-1})=P$. The same argument as above implies 
\begin{equation*}
    N^{P}(P\cap (hg)H_{2}(hg)^{-1})\cdot (P\cap g_{2}H_{1}g_{2}^{-1})=P
\end{equation*}
if $N^{G}(H_{2})H_{1}=G$. Therefore, we obtain $M_{2}(\cH_{P}^{\circ})\geq 2$. 

Now we can apply Proposition \ref{prop:tgqp}, and hence the $P$-lattice $J_{P/\cH_{P}^{\srd}}$ is not quasi-invertible. Consequently, the assertion follows from Theorem \ref{thm:qips}. 

(ii): By Proposition \ref{prop:syad}, there is a special subset $\cH_{0,P}^{\ast}$ of $\cH_{0,P}^{\set}$ such that $\#\cH_{0,P}^{\ast}=3$. 
Furthermore, by Lemma \ref{lem:psrd} (ii), we can take a special subset $\cH_{0,P}^{\circ}$ of $\cH_{P}^{\srd}$ satisfying $\#\cH_{0,P}^{\circ}=3$. In particular, we have $\#\cH_{P}^{\srd}\geq 3$. Hence, Proposition \ref{prop:tgqp} implies that the $P$-lattice $J_{P/\cH_{P}^{\srd}}$ is not quasi-invertible. Therefore, Theorem \ref{thm:qips} gives the desired assertion. 
\end{proof}

\begin{thm}\label{thm:gcdb}
Let $G$ be a finite group, and $\cH$ a reduced set of its subgroups satisfying $\#\cH \geq 2$ and $d_{G}(\cH)\geq 3$. We further assume that there is a special subset $\cH_{0}$ of $\cH$ with $\#\cH_{0}=2$. Then the $G$-lattice $J_{G/\cH}$ is not quasi-invertible. 
\end{thm}


\begin{proof}
Pick a prime divisor $p$ of $d_{G}(\cH)$, with $p>2$ if $d_{G}(\cH)$ is not a power of $2$. Fix a $p$-Sylow subgroup $P$ of $G$. Then Proposition \ref{prop:syad} implies that $\cH_{0,P}^{\set}$ admits its special subset $\cH_{0,P}^{\ast}$ satisfying $\#\cH_{0,P}^{\ast}=2$ and $\mu_{p}(\cH_{0,P}^{\ast})=\mu_{p}(\cH_{0})$. 
Furthermore, by Lemma \ref{lem:psrd} (ii), we can take a special subset $\cH_{0,P}^{\circ}$ of $\cH_{P}^{\srd}$ such that $\#\cH_{0,P}^{\circ}=2$. In particular, we have $\#\cH_{P}^{\srd}\geq 2$. 

\textbf{Case 1.~$p>2$ ($d_{G}(\cH)$ is not a power of $2$).}
In this case, Proposition \ref{prop:odqp} implies that the $P$-module $J_{P/\cH_{P}^{\srd}}$ is not quasi-invertible. Hence, the assertion follows from Theorem \ref{thm:qips}. 

\textbf{Case 2.~$p=2$ ($d_{G}(\cH)$ is a power of $2$). }
Since $d_{G}(\cH)>2$, Corollary \ref{cor:hppo} (i) implies $\mu_{2}(\cH_{P}^{\srd})\geq 2$. If $\#\cH_{P}^{\srd}\geq 3$, then Proposition \ref{prop:tgqp} implies that the $P$-lattice $J_{P/\cH_{P}}$ is not quasi-invertible. On the other hand, if $\#\cH_{P}^{\srd}=2$, then it is special. Hence, the same assertion on $J_{P/\cH_{P}^{\red}}$ as in case $\#\cH_{P}^{\srd}\geq 3$ holds by Proposition \ref{prop:tgqp}. Therefore, using Theorem \ref{thm:qips}, we obtain the desired assertion. 
\end{proof}

\begin{proof}[Proof of Theorem \ref{mtd2}]
Take a finite Galois field extension $L$ of $k$ that contains $K_{1},\ldots,K_{r}$. Put $G:=\Gal(L/k)$, $H_{i}:=\Gal(L/K_{i})$ for each $i\in \{1,\ldots,r\}$, and $\cH:=\{H_{i}<G\mid i\in \{1,\ldots,r\}\}$. By Proposition \ref{prop:coch}, it suffices to prove that $J_{G/\cH}$ is not quasi-invertible. 

(i): By assumption, $\cH$ is a special set and $(G:H)\in 4\Z$ for some $H\in \cH$. In particular, Lemma \ref{lem:psrd} (i) implies that $\cH$ is reduced. Hence, the assertion follows from Theorem \ref{thm:gcdt} (i). 

(ii): Take a subset $I$ of $\{1,\ldots,r\}$ with $\#I=3$ such that $L_{i}\cap K_{j}=k$ or $L_{j}\cap K_{i}=k$ for any $i,j\in I$ with $i\neq j$. Set $\cH_{0}:=\{H_{i}<G\mid i\in I\}$. Then it is a special subset of $\cH^{\set}$. Combining this with Lemma \ref{lem:psrd} (ii), we obtain that $\cH^{\red}$ admits its special subset $\cH'$ with $\#\cH'=3$. On the other hand, we have $[J_{G/\cH}]=[J_{G/\cH^{\red}}]$ by virtue of Proposition \ref{prop:rdrd}. Consequently, Theorem \ref{thm:gcdt} (ii) gives the desired assertion. 
\end{proof}

\begin{proof}[Proof of Theorem \ref{mtd3}]
The proof is similar to Theorem \ref{mtd2} (ii), except we use Theorem \ref{thm:gcdb} instead of Theorem \ref{thm:gcdt} (ii). 
\end{proof}

In the following, we discuss what happens if we relax the assumptions in Theorems \ref{thm:gcdt} and \ref{thm:gcdb}. 

\begin{ex}[{The case $d_{G}(\cH)=2$}]
The existence of a quasi-permutation $G$-lattice $J_{G/\cH}$ with $\#\cH=2$ that violates the assumption in Theorem \ref{thm:gcdt} (i) is given by Endo (\cite{Endo2001}). See also Proposition \ref{prop:tgqp}. In the following, we construct a quasi-invertible $G$-lattice $J_{G/\cH}$ such that $\#\cH=3$ and the assumption in Theorem \ref{thm:gcdt} (ii) fails. Let $p$ and $\ell$ be odd prime numbers that are distinct from each other. Set 
\begin{equation*}
    G:=C_{2p\ell} \times C_{2}=\langle \sigma,\tau \mid \sigma^{2p\ell}=\tau^{2}=1,\sigma\tau=\tau\sigma\rangle
\end{equation*}
and define its subgroups as follows: 
\begin{equation*}
    H_{1}:=\langle \sigma^{p} \rangle,\quad H_{2}:=\langle\sigma^{\ell} \rangle,\quad
    H_{3}:=\langle \sigma^{2}, \tau \rangle. 
\end{equation*}
Put $\cH:=\{H_{1},H_{2},H_{3}\}$. Then we have $d_{G}(\cH)=2$. Moreover, the following hold: 
\begin{equation*}
    H_{1}H_{3}=H_{2}H_{3}=G,\quad H_{1}H_{2}=\langle \sigma \rangle\neq G,\quad (H_{1}\cap H_{2})H_{3}=G. 
\end{equation*}
Note that the left and central equations imply that $\cH$ is not special. In this case, the $G$-lattice $J_{G/\cH}$ is quasi-permutation, which can be confirmed as follows. Put $G_{p}:=\langle \sigma^{2\ell}\rangle$, $G_{\ell}:=\langle \sigma^{2p}\rangle$ and $G_{2}:=\langle \sigma^{p\ell},\tau \rangle$. Then we have $G=G_{p}\times G_{\ell} \times G_{2}$. Moreover, the following hold: 
\begin{equation*}
    \cH_{G_{p}}^{\red}=\{G_{p}\},\quad 
    \cH_{G_{\ell}}^{\red}=\{G_{\ell}\},\quad 
    \cH_{G_{2}}^{\red}=\{\langle \sigma^{p\ell}\rangle,\langle \tau \rangle\}. 
\end{equation*}
Applying \cite[Lemma 8.3]{Hasegawa}, we obtain an equality in $\sS(G)$ as follows: 
\begin{equation*}
    [J_{G/\cH}]=[J_{G/\cH'}], 
\end{equation*}
Here $\cH'$ is a reduced set satisfying $(\cH')_{G_{\ell'}}^{\set}=\cH_{G_{\ell'}}^{\red}$ for any $\ell'\in \{2,p,\ell\}$. Then one has an equality
\begin{equation*}
    \cH'=\{G_{p}\times G_{\ell}\times \langle \sigma^{p\ell}\rangle, G_{p}\times G_{\ell}\times \langle \tau \rangle\}. 
\end{equation*}
In particular, we have $G/N^{G}(\cH')\cong D_{2}$ and $\#\cH'=2$. Hence, \cite[Theorem 8.4]{Hasegawa} implies that $J_{G/\cH}$ is quasi-permutation. 
\end{ex}

\begin{ex}[{The case $d_{G}(\cH)\geq 3$}]
Let $p>2$ be an odd prime number. Take prime numbers $\ell_{1}$ and $\ell_{2}$ so that $\ell_{1}\equiv 1\bmod p$ and $\ell_{2}\notin \{p,\ell_{1}\}$. Pick an element $m\in (\Z/\ell_{1})^{\times}$ of order $p$, and set $G:=\langle \sigma_{1},\sigma_{2},\tau_{1},\tau_{2}\rangle$, where
\begin{gather*}
    \sigma_{1}^{\ell_{1}}=\sigma_{2}^{\ell_{2}}=\tau_{1}^{p}=\tau_{2}^{p}=1,\quad 
    \sigma_{1}\sigma_{2}=\sigma_{2}\sigma_{1},\quad 
    \tau_{1}\tau_{2}=\tau_{2}\tau_{1},\\
    \tau_{1}\sigma_{1}\tau_{1}^{-1}=\sigma_{1}^{m},\quad
    \tau_{j}\sigma_{i}=\sigma_{i}\tau_{j}\text{ for }(i,j)\neq (1,1). 
\end{gather*}
Then there is an isomorphism $G\cong (C_{\ell_{1}\ell_{2}}\rtimes C_{p})\times C_{p}$. Put $\cH:=\{H_{1},H_{2}\}$, where
\begin{equation*}
    H_{1}:=\langle \sigma_{1}, \tau_{1}\rangle,\quad 
    H_{2}:=\langle \sigma_{2}, \tau_{1}\rangle. 
\end{equation*}
Then, $H_{1}$ and $H_{1}H_{2}$ are normal in $G$. Moreover, one has $d_{G}(\cH)=p$ and $(G:H_{1}H_{2})=p$. In particular, there is no special subset of $\cH$. In this case, the $G$-lattice $J_{G/\cH}$ is quasi-invertible. Indeed, let $P:=\langle \tau_{1},\tau_{2}\rangle$, which is a $p$-Sylow subgroup of $G$. Then the set $\cH_{P}^{\red}$ consists of $\langle \tau_{1} \rangle$ since it coincides with $P\cap H_{1}$ and $P\cap H_{2}$. Hence, we obtain an isomorphism of $P$-modules $J_{P/\cH_{P}^{\red}}\cong J_{P/\langle \tau_{1}\rangle}$, which is quasi-permutation by Proposition \ref{prop:edgp}. Combining this result with Theorem \ref{thm:qips}, we obtain that $J_{G/\cH}$ is quasi-invertible as desired. 
\end{ex}

\section{Proof of Theorem \ref{mth1}}\label{sect:pfzg}

We first recall previous results on Endo--Miyata (\cite{Endo1975}) and Endo (\cite{Endo2011}) using the terminology in Section \ref{sect:rvrt}. 

\begin{prop}[{\cite[Theorem 1.5]{Endo1975}}]\label{prop:emqi}
Let $G$ be a finite group. Then the following are equivalent: 
\begin{enumerate}
\item all Sylow subgroups of $G$ are cyclic; 
\item all coflabby $G$-lattices are invertible; 
\item all $G$-lattices are quasi-invertible; 
\end{enumerate}
\end{prop}

\begin{prop}[{\cite[Theorem 2.3]{Endo1975}}]\label{prop:edgp}
Let $G$ be a finite group. Then the following are equivalent: 
\begin{enumerate}
\item $J_{G}$ is a quasi-permutation $G$-lattice; 
\item there is an isomorphism $G\cong C_{m}\rtimes C_{2^{\nu}}$, where $m$ is an odd integer, $\nu \in \Znn$, and the action of $C_{2^{\nu}}$ on $C_{m}$ factors through a homomorphism $C_{2^{\nu}}\rightarrow C_{2}$. 
\end{enumerate}
\end{prop}

\begin{prop}[{\cite[Theorem 3.1]{Endo2011}}]\label{prop:edqp}
Let $G$ be a finite group of which all Sylow subgroups are cyclic. Consider a non-trivial subgroup $H$ of $G$ satisfying $N^{G}(H)=\{1\}$. Then the following are equivalent: 
\begin{enumerate}
\item $J_{G/H}$ is a quasi-permutation $G$-lattice; 
\item $G\cong C_{m}\times D_{n}$ with $m,n\notin 2\Z$, $\gcd(m,n)=1$ and $\#H=2$. 
\end{enumerate}
\end{prop}

In the following, we consider $J_{G/\cH}$ when all Sylow subgroups of $G$ are cyclic. 

\begin{lem}[{\cite[Lemma 6.3 (iii)]{Oki}}]\label{lem:cjsd}
Let $p$ be a prime number, and $G=G_{p}\rtimes G'$. Here $G_{p}$ is a finite abelian $p$-group, and $G'$ is a finite group of order coprime to $p$. Then, for any subgroup $H$ of $G$, there exist $s\in G_{p}$ and a subgroup $H'$ of $G'$ so that $sHs^{-1}=(G_{p}\cap H)\rtimes H'$. 
\end{lem}

\begin{lem}\label{lem:hall}
Let $G$ be a finite group of which all Sylow subgroups are cyclic. Then there exists an isomorphism
\begin{equation*}
G\cong N\rtimes G',
\end{equation*}
where $N$ and $G'$ are cyclic, satisfying \emph{(i)} and \emph{(ii)} as follows. 
\begin{enumerate}
\item The integers $\#N$ and $\#G'$ are coprime to each other. 
\item For any prime divisor $p$ of $\#G$, a $p$-Sylow subgroup of $G$ is normal if and only if $p\mid \#N$. 
\end{enumerate}
\end{lem}

\begin{proof}
By \cite[Theorem 9.4.3]{Hall1959}, there exist elements $\sigma,\tau$ of $G$ and positive integers $m,n$ and $s\in \Z$, where $m\notin 2\Z$, $s^{n}\equiv 1\bmod m$, $s\in \{0,\ldots,m-1\}$ and $\gcd(m,n)=\gcd(m,s-1)=1$, so that 
\begin{equation*}
    G=\langle \sigma,\tau\mid \sigma^{m}=\tau^{n}=1,\tau\sigma\tau^{-1}=\sigma^{s}\rangle. 
\end{equation*}
On the other hand, for a prime divisor $\ell$ of $n$, an $\ell$-Sylow subgroup of $G$ is normal if and only if the following hold: 
\begin{equation}\label{eq:pdnd}
    \ell\mid n/\#\langle s\bmod m\rangle,\quad \ell \nmid \#\langle s\bmod m\rangle. 
\end{equation}
Write $n'$ for the product of $\ell^{\ord_{\ell}(n)}$ for all prime numbers $\ell$ that satisfy \eqref{eq:pdnd}. Let $N$ be the subgroup of $G$ generated by $\sigma$ and $\tau^{n/n'}$. We denote by $G'$ the subgroup of $G$ generated by $\tau^{n'}$. Then $N$ and $G'$ satisfy the desired conditions. 
\end{proof}

Let $G$ be a finite group of which all Sylow subgroups are cyclic.
\begin{itemize}
\item We denote by $\sP(G)$ the set of prime divisors $p$ of $\#G$ such that a $p$-Sylow subgroup of $G$ is normal. 
\item For $p\in \sP(G)$, let $\alpha_{G,p}\colon G \rightarrow \Aut(G_{p});g \mapsto [h\mapsto ghg^{-1}]$. 
\item Set $\sP_{*}(G):=\{p\in \sP(G)\mid \#\Ima(\alpha_{G,p})\geq 3\}$. 
\end{itemize}
On the other hand, for a finite abelian group $A$ and a prime number $p$, set
\begin{equation*}
A^{(p)}:=\{a\in A\mid na=0\text{ for some }n\in \Z\setminus p\Z\}. 
\end{equation*}

\begin{lem}\label{lem:nmqt}
Let $G$ be a finite group of which all Sylow subgroups are cyclic, and $p\in \sP(G)$. Take a $p$-subgroup $N_{p}$ with $p\mid (G:N_{p})$, which is normal in $G$. Consider the commutative diagram
\begin{equation}\label{eq:dgqt}
\begin{gathered}
\xymatrix@C=35pt{
G\ar[r]^{\alpha_{G,p}}\ar[d]_{g\mapsto gN_{p}}&\Aut(G_{p})\ar[d]^{\beta \mapsto [gN_{p}\mapsto \beta(g)N_{p}]}\\
G/N_{p} \ar[r]^{\alpha_{G/N_{p},p}\hspace{18pt}}&\Aut(G_{p}/N_{p}). }
\end{gathered}
\end{equation}
Then the right vertical map in \eqref{eq:dgqt} induces an isomorphism $\Ima(\alpha_{G,p})\cong \Ima(\alpha_{G/N_{p},p})$. 
\end{lem}

\begin{proof}
Since $G_{p}$ is cyclic of order a power of $p$, the right vertical map in \eqref{eq:dgqt} induces an isomorphism
\begin{equation*}
\Aut(G_{p})^{(p)} \cong \Aut(G_{p}/N_{p})^{(p)}. 
\end{equation*}
Hence, the assertion holds. 
\end{proof}

For a multiset $\cH$ of subgroups of a finite group $G$, we denote by $\sD_{G}(\cH)$ the set of prime divisors of $d_{G}(\cH)$. If $\cH=\{H_{0}\}$ with $H_{0}<G$, we simply write $\sD_{G}(H_{0})$ for $\sD_{G}(\cH)$. 

\begin{lem}\label{lem:pgcd}
Let $G$ be a finite group of which all Sylow subgroups are cyclic, and $H_{0}$ its subgroup. Put $\overline{G}:=G/N^{G}(H_{0})$, and consider a prime divisor $p$ of $\#G$. Then we have 
\begin{equation*}
\sP_{*}(\overline{G})=\sD_{G}(H_{0})\cap \sP_{*}(G). 
\end{equation*}
\end{lem}

\begin{proof}
We first demonstrate the following: 
\begin{equation*}
\sP_{*}(\overline{G})\subset \sP(G). 
\end{equation*}
Assume that a prime divisor $p$ of $\#G$ satisfies $p\notin \sP(G)$. It suffices to prove that $p\in \sP(\overline{G})$ implies $p\notin \sP_{*}(G)$. Write $G=N\rtimes G'$, where $N$ and $G'$ are as in Lemma \ref{lem:hall}. Put $\overline{N}:=NN^{G}(H_{0})/N^{G}(H_{0})$ and $\overline{G}':=G'N^{G}(H_{0})/N^{G}(H_{0})$, which satisfy $\overline{G}=\overline{N}\overline{G}'$. Denote by $\overline{G}_{p}$ the $p$-Sylow subgroup of $\overline{G}$. Then, $p\in \sP(\overline{G})$ implies that $\overline{G}_{p}\overline{N}$ is abelian. On the other hand, since $p\notin \sP(G)$, $\overline{G}_{p}$ is contained in $\overline{G}'$. Hence, $\overline{G}_{p}\overline{G}'$ is also abelian. Therefore, $\overline{G}_{p}$ is central in $\overline{G}$. Consequently, we have $p\notin \sP_{*}(\overline{G})$ as desired. 

\vspace{6pt}
Now, take $p\in \sP(G)$. Let $G_{p}$ be the unique $p$-Sylow subgroup of $G$. Then $G_{p}\cap H_{0}$ is normal in $G$ because $G_{p}$ is cyclic. Hence, we have $G_{p}\cap H_{0}\subset N^{G}(H_{0})$, and this implies $G_{p}\cap H_{0}=G_{p}\cap N^{G}(H_{0})$. Consequently, one has
\begin{equation*}
G_{p}/(G_{p}\cap H_{0})\cong G_{p}N^{G}(H_{0})/N^{G}(H_{0}). 
\end{equation*}
Hence, we have $p\in \sP(\overline{G})$ if and only if $p\in \sD_{G}(H_{0})\cap \sP(G)$. Therefore, the assertion follows from Lemma \ref{lem:nmqt}. 
\end{proof}

Now, Theorem \ref{mth1} follows from Proposition \ref{prop:coch} and Theorem \ref{thm:emmn} as follows. 

\begin{thm}\label{thm:emmn}
Let $G$ be a finite group of which all Sylow subgroups are cyclic, and $\cH$ a reduced set of its subgroups. Then the $G$-lattice $J_{G/\cH}$ is quasi-invertible. Moreover, it is quasi-permutation if and only if 
\begin{equation*}
\sD_{G}(\cH)\cap \sP_{*}(G)=\emptyset. 
\end{equation*}
\end{thm}

\begin{proof}
Write $G=N\rtimes G'$, where $N$ and $G'$ are as in Lemma \ref{lem:hall}. By Lemma \ref{lem:cjsd}, we may assume that all subgroups of $G$ lying in $\cH$ are of the form $C\rtimes H'$, where $C<N$ and $H'<G'$. 

\textbf{Case 1.~$\#\cH=1$. }

Write $\cH=\{H_{0}\}$. Put $\overline{G}:=G/N^{G}(H_{0})$ and $\overline{H}_{0}:=H_{0}/N^{G}(H_{0})$. By Corollary \ref{cor:trfx}, there is an isomorphism of $G$-lattices
\begin{equation*}
J_{\overline{G}/\overline{H}_{0}}\cong J_{G/H_{0}}. 
\end{equation*}
Hence, by Lemma \ref{lem:pgcd}, it suffices to prove that the $G$-lattice $J_{G/H_{0}}$ is quasi-permutation if and only if $\sP_{*}(G)=\emptyset$, under the additional assumption $N^{G}(H_{0})=\{1\}$. If $H_{0}=\{1\}$, then the assertion follows from Proposition \ref{prop:edgp}. Otherwise, Proposition \ref{prop:edqp} gives the desired assertion. 

\vspace{6pt}
\textbf{Case 2.~$\#\cH \geq 2$. }

We first note that $\sP(G)$ coincides with the set of prime divisors of $\#N$. We denote by $H_{0}$ the subgroup of $G$ generated by $H$ for all $H\in \cH$. Let 
\begin{equation*}
d_{1}(\cH):=\prod_{p\in \sP(G)}p^{\ord_{p}d_{G}(\cH)},\quad 
d'_{1}(\cH):=\prod_{p\notin \sP(G)}p^{\ord_{p}d_{G}(\cH)}. 
\end{equation*}
Then we have $H_{0}=C_{0}\rtimes H'_{0}$, where $(N:C_{0})=d_{1}(\cH)$ and $(G':H'_{0})=d'_{1}(\cH)$. In particular, one has $(G:H_{0})=d_{G}(\cH)$. Therefore, Theorem \ref{thm:bigh} gives an equality
\begin{equation*}
[J_{G/\cH}]=[J_{G/H_{0}}]
\end{equation*}
in $\sS(G)$. Hence, the assertion is reduced to Case 1 since $\sD_{G}(H_{0})=\sD_{G}(\cH)$ holds. 
\end{proof}

\begin{ex}
Consider the finite group as follows: 
\begin{equation*}
G:=\langle \sigma,\tau \mid \sigma^{15}=\tau^{4}=1,\tau \sigma \tau^{-1}=\sigma^{2}\rangle. 
\end{equation*}
Note that all Sylow subgroups of $G$ are cyclic. Moreover, the homomorphism
\begin{equation*}
\alpha \colon \langle \tau \rangle \rightarrow \Aut(\langle \sigma \rangle)
\end{equation*}
is injective, and hence $i(G):=\#\Ima(\alpha)=4$. Note that the symbol $i(G)$ is used in \cite{Endo1975} and \cite{Endo2011}. Moreover, one has $\sP_{*}(G)=\{5\}$. Let $\cH:=\{\langle \sigma^{3},\tau^{2}\rangle,\langle \tau \rangle\}$, which is a reduced set of subgroups of $G$. Then one has $d_{G}(\cH)=\gcd(6,15)=3$. Hence $\sD_{G}(\cH)=\{3\}$, in particular, it does not intersect with $\sP_{*}(G)$. Therefore, the $G$-lattice $J_{G/\cH}$ is quasi-permutation by Theorem \ref{thm:emmn}. 
\end{ex}

\section{Coflabby resolution of particular lattices over dihedral groups}\label{sect:cflb}

This section is devoted to the proof of the following. 

\begin{thm}\label{thm:nzf2}
Let $m$ be an odd positive integer, and consider a strongly reduced set 
\begin{equation*}
\cH_{m}:=\{\langle \sigma_{2m}^{m}\rangle, \langle \tau_{2m} \rangle \}
\end{equation*}
of subgroups of $D_{2m}$. Then the $D_{2m}$-lattice $J_{D_{2m}/\cH_{m}}$ is quasi-permutation. 
\end{thm}

In the rest of this section, write $I_{m}=I_{D_{2m}/\cH_{m}}$, and regard $I_{m}$ as a $D_{2m}$-sublattice of the direct sum $\Z[D_{2m}/\langle \sigma_{2m}^{m}\rangle]\oplus \Z[D_{2m}/\langle \tau_{2m}\rangle]$. Moreover, put
\begin{equation*}
D^{(m)}:=\langle \sigma_{2m}^{2},\tau_{2m}\rangle, 
\end{equation*}
which is a subgroup of index $2$ in $D_{2m}$. Then, our strategy for the proof is as follows. 
\begin{enumerate}
\item[Step 1:\hspace{-20pt}] \hspace{20pt}Construct a specific exact sequence of $D_{2m}$-lattices
\begin{equation}\label{eq:exsq}
    0\rightarrow E_{m}\rightarrow R_{m}\rightarrow I_{m}\rightarrow 0,
\end{equation}
 \hspace{20pt} which should be a coflabby resolution of $I_{m}$ (Corollary \ref{cor:eris} and Proposition \ref{prop:cfcf}). 
\item[Step 2:\hspace{-20pt}] \hspace{20pt} Prove that $E_{m}$ is invertible (Proposition \ref{prop:fmiv}). 
\item[Step 3:\hspace{-20pt}] \hspace{20pt} Confirm that $E_{m}\oplus \Z[D_{2m}/\langle \sigma_{2m}^{2},\sigma_{2m}\tau_{2m} \rangle]$ is a permutation $D_{2m}$-lattice (Proposition \ref{prop:emsp}).
\item[Step 4:\hspace{-20pt}] \hspace{20pt} The dual of the exact sequence 
    \begin{equation*}
        0\rightarrow E_{m}\oplus \Z[D_{2m}/\langle \sigma_{2m}^{2},\sigma_{2m}\tau_{2m}\rangle]\rightarrow R_{m}\oplus \Z[D_{2m}/\langle \sigma_{2m}^{2},\sigma_{2m}\tau_{2m}\rangle] \rightarrow I_{m}\rightarrow 0
\end{equation*}
 \hspace{20pt} induced by \eqref{eq:exsq} implies the desired assertion. 
\end{enumerate}

We first perform Step 1. Consider three homomorphisms of $D_{2m}$-lattices as follows: 
\begin{gather*}
\psi_{m,1}\colon R_{m,1}:=\Z[D_{2m}]\rightarrow I_{m};\,1\mapsto (1,-1);\\
\psi_{m,2}\colon R_{m,2}:=\Z[D_{2m}/\langle \sigma_{2m}^{m}\rangle]\rightarrow I_{m};\,1\mapsto (1+\sigma_{2m}^{2},-(1+\sigma_{2m}^{m}));\\
\psi_{m,3}\colon R_{m,3}:=\Z \rightarrow I_{m};\,1\mapsto ((1+\cdots+\sigma_{2m}^{m-1})(1+\tau_{2m}),-(1+\cdots+\sigma_{2m}^{2m-1})). 
\end{gather*}
Define a $\Z[D_{2m}]$-homomorphism $\Psi_{m}$ from $R_{m}:=R_{m,1}\oplus R_{m,2} \oplus R_{m,3}$ to $I_{m}$ to be
\begin{equation*}
R_{m} \rightarrow I_{m};\,(x_{1},x_{2},x_{3})\mapsto \psi_{m,1}(x_{1})+\psi_{m,2}(x_{2})+\psi_{m,3}(x_{3}). 
\end{equation*}
On the other hand, we denote by $I_{m,0}$ the image of $I_{D_{2m}/\langle \sigma_{2m}^{m}\rangle}$ under the canonical injection
\begin{equation*}
\Z[D_{2m}/\langle \sigma_{2m}^{m}\rangle] \hookrightarrow \Z[D_{2m}/\langle \sigma_{2m}^{m}\rangle] \oplus \Z[D_{2m}/\langle \tau_{2m}\rangle];\,x\mapsto (x,0). 
\end{equation*}
It is clear that $I_{m,0}$ is contained in $I_{m}$. 

\begin{lem}\label{lem:zfsj}
The image of $R_{m}^{\langle \sigma_{2m}^{m}\rangle}$ under $\Psi_{m}$ contains $I_{m,0}$. 
\end{lem}

\begin{proof}
By definition, we have the following: 
\begin{gather*}
\Psi_{m}(1+\sigma_{2m}^{m},-1,0)=(1-\sigma_{2m}^{2},0),\\
\Psi_{m}((1+\sigma_{2m}^{m})(1-\tau_{2m}),0,0)=(2(1-\tau_{2m}),0),\\
\Psi_{m}\left((1+\sigma_{2m}^{m})\left(\sum_{i=0}^{m-1}\sigma_{2m}^{2i}\right),0,-1\right)=\left(\left(\sum_{i=0}^{m-1}\sigma_{2m}^{2i}\right)(1-\tau_{2m}),0\right). 
\end{gather*}
Since $m$ is odd, we obtain that $(1-\sigma_{2m},0)$ and $(1-\tau_{2m},0)$ lie in $\Ima(\Psi_{m})$. Hence, the assertion follows from the equality $I_{m,0}=\langle (1-\sigma_{2m},0),(1-\sigma_{2m}^{m-2}\tau_{2m},0)\rangle_{\Z[D_{2m}]}$. 
\end{proof}

By Lemma \ref{lem:zfsj}, we obtain the following. 

\begin{cor}\label{cor:eris}
We define $E_{m}$ as the kernel of $\Psi_{m}$. Then, there is an exact sequence of $D_{2m}$-lattices
\begin{equation}\label{eq:cdcf}
0\rightarrow E_{m} \rightarrow R_{m} \xrightarrow{\Psi_{m}} I_{m} \rightarrow 0. 
\end{equation}
\end{cor}

From here, we move on to Step 2. To accomplish it, we demonstrate that \eqref{eq:cdcf} is a coflabby resolution, and use Theorem \ref{thm:qips}. 

\begin{lem}[{\cite[Lemma 6.4]{Hasegawa}}]\label{lem:lmis}
Consider a commutative diagram of finite free abelian groups
\begin{equation*}
\xymatrix{
0\ar[r]& M_{1}\ar[r] \ar[d]^{f_{1}} & M_{2} \ar[r]\ar[d]^{f_{2}}& M_{3} \ar[r]\ar[d]^{f_{3}}& 0\\
0\ar[r]& M'_{1}\ar[r] & M'_{2} \ar[r] & M'_{3}, &
}
\end{equation*}
where the horizontal sequences are exact. We further assume that 
\begin{itemize}
    \item $\rk_{\Z}(M_{1})=\rk_{\Z}(M'_{1})$; 
    \item $f_{2}$ and $f_{3}$ are injective; and 
    \item the cokernel of $f_{2}$ is torsion-free. 
\end{itemize}
Then the homomorphism $f_{1}$ is an isomorphism. 
\end{lem}

The following will be used not only in Step 2 but also in Step 3. 

\begin{lem}\label{lem:rsdg}
We define an element $x_{1}$ of $R_{m}$ as follows: 
\begin{equation*}
x_{1}:=
\begin{cases}
\left(1+\sigma_{2m}^{m}\tau_{2m},(1+\tau_{2m})\sum_{j=1}^{(m-1)/4}(\sigma_{2m}^{(-1)^{j-1}(2j-1)}+\sigma_{2m}^{(-1)^{j-1}\cdot 2j}),-1\right)&\text{if }m\equiv 1\bmod 4;\\
\left(1+\sigma_{2m}^{m}\tau_{2m},-(1+\tau_{2m})\sum_{j=0}^{(m-3)/4}(\sigma_{2m}^{(-1)^{j}\cdot 2j}+\sigma_{2m}^{(-1)^{j}(2j+1)}),1\right)&\text{if }m\equiv 3\bmod 4. 
\end{cases}
\end{equation*}
\begin{enumerate}
\item We have $x_{1}\in E_{m}^{\langle \sigma_{2m}^{m}\tau_{2m} \rangle}$ and $\langle x_{1}\rangle_{\Z[D_{2m}]}\cong \Z[D_{2m}/\langle \sigma_{2m}^{m}\tau_{2m} \rangle]$. 
\item There is an isomorphism of $D^{(m)}$-lattices
\begin{equation*}
E_{m}\cong \Z[D^{(m)}/\langle \sigma_{2m}^{2}\rangle]\oplus \Z[D^{(m)}]
\end{equation*}
such that the composite
\begin{equation*}
\Z[D^{(m)}]\xrightarrow{c\mapsto (0,c)} \Z[D^{(m)}/\langle \sigma_{2m}^{2}\rangle]\oplus \Z[D^{(m)}]\cong E_{m}
\end{equation*}
is given by $1\mapsto x_{1}$. In particular, $E_{m}$ is permutation as a $D^{(m)}$-lattice, and 
\begin{equation*}
E_{m}=E_{m}^{\langle \sigma_{2m}^{2}\rangle}+\langle x_{1}\rangle_{\Z[D_{2m}]}. 
\end{equation*}
\end{enumerate}
\end{lem}

\begin{proof}
(i): The assertion $x_{1}\in E_{m}^{\langle \sigma_{2m}^{m}\tau_{2m} \rangle}$ follows from direct computation. In particular, we obtain a surjection
\begin{equation}\label{eq:xogn}
\Z[D_{2m}/\langle \sigma_{2m}^{m}\tau_{2m} \rangle]\twoheadrightarrow \langle x_{1}\rangle_{\Z[D_{2m}]};\,1\mapsto x_{1}. 
\end{equation}
On the other hand, since $x_{1}\in E_{m}$, the canonical projection $R_{m}\twoheadrightarrow R_{m,1}$ induces a surjection
\begin{equation}\label{eq:fmds}
\langle x_{1}\rangle_{\Z[D_{2m}]} \twoheadrightarrow \langle 1+\sigma_{2m}^{m}\tau_{2m}\rangle_{\Z[D_{2m}]};\,x_{1}\mapsto 1+\sigma_{2m}^{m}\tau_{2m}. 
\end{equation}
The right-hand side of \eqref{eq:fmds} is isomorphic to $\Z[D_{2m}/\langle \sigma_{2m}^{m}\tau_{2m}\rangle]$. Hence \eqref{eq:xogn} is an isomorphism, which concludes the second assertion. 

(ii): Consider three isomorphisms of $D^{(m)}$-lattices as follows: 
\begin{gather*}
r_{m,1}\colon \Z[D^{(m)}]^{\oplus 2}\xrightarrow{\cong}\Z[D_{2m}];\,(a_{1},a_{2})\mapsto a_{1}+\sigma_{2m}^{m}a_{2};\\
r_{m,2}\colon \Z[D^{(m)}]\xrightarrow{\cong} \Z[D_{2m}/\langle \sigma_{2m}^{m}\rangle];\,b\mapsto b;\\
r_{m,3}\colon \Z[D^{(m)}/\langle \tau_{2m}\rangle]^{\oplus 2}\xrightarrow{\cong}\Z[D_{2m}/\langle \tau_{2m}\rangle];\,(c_{1},c_{2})\mapsto c_{1}+\sigma_{2m}^{m}c_{2}. 
\end{gather*}
Put $I'_{m}:=I_{D^{(m)}/\cH'_{m}}$, where $\cH'_{m}:=\{\{1\},\langle \tau_{2m}\rangle,\langle \tau_{2m}\rangle\}$. Then $r_{m,1}$, $r_{m,2}$ and $r_{m,3}$ induce commutative diagrams of $D^{(m)}$-lattices
\begin{gather*}
\xymatrix@C=115pt{
R_{m,1}^{(0)}\oplus R_{m,1}^{(1)}\ar[r]^{\hspace{20pt}(a_{1},a_{2})\mapsto \psi_{m,1}^{(0)}(a_{1})+\psi_{m,1}^{(1)}(a_{2})}\ar[d]^{\cong}&I'_{m}\ar[d]^{\cong}\\
R_{m,1}\ar[r]^{\psi_{m,1}}& I_{m},}\\
\xymatrix@C=45pt{
R'_{m,2} \ar[r]^{\psi'_{m,2}}\ar[d]^{\cong}&I'_{m}\ar[d]^{\cong}\\
R_{m,2} \ar[r]^{\psi_{m,2}}& I_{m},
}\quad
\xymatrix@C=45pt{
R'_{m,3} \ar[r]^{\psi'_{m,3}}\ar@{=}[d]&I'_{m}\ar[d]^{\cong}\\
R_{m,3}\ar[r]^{\psi_{m,3}}& I_{m}. 
}
\end{gather*}
Here $\psi_{m,1}^{(0)}$, $\psi_{m,1}^{(1)}$, $\psi'_{m,2}$ and $\psi'_{m,3}$ are defined as follows: 
\begin{gather*}
\psi_{m,1}^{(0)}\colon R_{m,1}^{(0)}:=\Z[D^{(m)}]\rightarrow I'_{m};\,1\mapsto (1,-1,0);\\
\psi_{m,1}^{(1)}\colon R_{m,1}^{(1)}:=\Z[D^{(m)}]\rightarrow I'_{m};\,1\mapsto (1,0,-1);\\
\psi'_{m,2}\colon R'_{m,2}:=\Z[D^{(m)}]\rightarrow I'_{m};\,1\mapsto (1+\sigma_{2m}^{2},-1,-1);\\
\psi'_{m,3}\colon R'_{m,3}:=\Z \rightarrow I'_{m};\,1\mapsto \left(\sum_{i=0}^{m-1}\sigma_{2m}^{2i}(1+\tau_{2m}),-\sum_{i=0}^{m-1}\sigma_{2m}^{2i},-\sum_{i=0}^{m-1}\sigma_{2m}^{2i}\right). 
\end{gather*}
\begin{claima}
The homomorphism $R_{m,1}^{(0)}\oplus R'_{m,2}\rightarrow I'_{m};\,(c_{1},c_{2})\mapsto \psi_{m,1}^{(0)}(c_{1})+\psi'_{m,2}(c_{2})$ is surjective. 
\end{claima}

\begin{proof}
It suffices to prove that $\Ima(\psi_{m,1}^{(0)})+\Ima(\psi'_{m,2})$ contains $(1-\tau_{2m},0,0)$ and $(1-\sigma_{2m}^{2},0,0)$. By direct computation, we have
\begin{equation*}
\psi_{m,1}^{(0)}(1-\tau_{2m})=(1-\tau_{2m},0,0),\ \psi'_{m,2}(\sigma_{2m}^{2}(\tau_{2m}-1))-\psi_{m,1}^{(0)}((1+\sigma_{2m}^{2})(\tau_{2m}-1))=(1-\sigma_{2m}^{4},0,0). 
\end{equation*}
Since $2\bmod m$ is a unit in $\Z/m$, we obtain $(1-\sigma_{2m}^{2},0,0)\in \Ima(\psi_{m,1}^{(0)})+\Ima(\psi'_{m,2})$. 
\end{proof}

By Claim, $\psi_{m,1}^{(0)}$, $\psi'_{m,2}$ and $\psi'_{m,3}$ induce an exact sequence of $D^{(m)}$-lattices
\begin{equation*}
0\rightarrow E'_{m} \rightarrow R'_{m} \xrightarrow{\Psi'_{m}} I'_{m} \rightarrow 0, 
\end{equation*}
where $R'_{m}:=R_{m,1}^{(0)}\oplus R'_{m,2}\oplus R'_{m,3}$. Then we have $E'_{m}=E_{m}\cap R'_{m}$ and $\rk_{\Z}(E'_{m})=2$. Furthermore, Claim (a) for $m=1$ implies that $\Psi'_{m}$ induces a surjection $(R'_{m})^{\langle \sigma_{2m}^{2}\rangle}\twoheadrightarrow (I'_{m})^{\langle \sigma_{2m}^{2}\rangle}$. In particular, there is a commutative diagram
\begin{equation*}
\xymatrix{
0\ar[r]&(E'_{m})^{\langle \sigma_{2m}^{2}\rangle}\ar[r]\ar[d]&(R'_{m})^{\langle \sigma_{2m}^{2}\rangle}\ar[r]^{\Psi'_{m}}\ar[d]&(I'_{m})^{\langle \sigma_{2m}^{2}\rangle}\ar[r]\ar[d]&0\\
0\ar[r]&E'_{m}\ar[r]&R'_{m}\ar[r]^{\Psi'_{m}}&I'_{m}\ar[r]&0,
}
\end{equation*}
where the central and rightmost vertical maps are injective. Note that the cokernel of the central vertical map is torsion-free. Hence, Lemma \ref{lem:lmis} implies $E'_{m}=(E'_{m})^{\langle \sigma_{2m}^{2}\rangle}$. 
\begin{claimb}
There is an isomorphism of $D^{(m)}$-lattices $(E'_{m})^{\langle \sigma_{2m}^{2}\rangle}\cong \Z[D^{(m)}/\langle \sigma_{2m}^{2}\rangle]$. 
\end{claimb}

\begin{proof}
Let 
\begin{equation*}
x_{0}:=\left(\left(\sum_{i=0}^{m-1}\sigma_{2m}^{2i}\right)(\tau_{2m}-1),\sum_{i=0}^{m-1}\sigma_{2m}^{2i},-1\right)\in R'_{m}. 
\end{equation*}
Then it lies in $(E'_{m})^{\langle \sigma_{2m}^{2}\rangle}$. This implies that the natural projection $R'_{m}\twoheadrightarrow R'_{m,2}$ induces a surjection 
\begin{equation*}
E'_{m} \twoheadrightarrow \Z[D^{(m)}/\langle \sigma_{2m}^{2} \rangle]. 
\end{equation*}
It is an isomorphism since $\rk_{\Z}(E'_{m})=2$. This completes the proof of Claim (b). 
\end{proof}

Since $E'_{m}=E_{m}\cap R'_{m}$, one has an exact sequence of $D^{(m)}$-lattices
\begin{equation}\label{eq:fcfp}
0\rightarrow E'_{m}\rightarrow E_{m}\rightarrow \Z[D^{(m)}]\rightarrow 0. 
\end{equation}
Then $x_{1}$ produces a right splitting of \eqref{eq:fcfp}. In particular, we obtain an isomorphism of $D^{(m)}$-lattices
\begin{equation*}
E_{m}\cong E'_{m}\oplus \Z[D^{(m)}]. 
\end{equation*}
Hence, the proof is complete. 
\end{proof}

\begin{lem}\label{lem:flpp}
Let $G$ be a finite group, and
\begin{equation}\label{eq:abcf}
0\rightarrow E \rightarrow R \rightarrow M \rightarrow 0
\end{equation}
an exact sequence of $G$-lattices with $R$ permutation. Then it is a coflabby resolution of $M$ if and only if 
\begin{itemize}
    \item[(a)] $H^{1}(H,E)=0$; or
    \item[(b)] $R^{H}\rightarrow M^{H}$ is surjective,
\end{itemize}
for all subgroups of prime power orders in $G$. 
\end{lem}

\begin{proof}
Take a subgroup $H$ of $G$. For each prime divisor $p$ of $\#H$, pick a $p$-Sylow subgroup $H_{p}$ of $H$. Then, the homomorphism
\begin{equation*}
    (\Res_{H/H_{p}})_{p}\colon H^{1}(H,E)\rightarrow \bigoplus_{p\mid \#H}H^{1}(H_{p},E)
\end{equation*}
is injective. Hence, we have $H^{1}(H,E)=0$ if and only if $H^{1}(H_{p},E)=0$ for any prime divisor $p$ of $\#H$. Therefore, \eqref{eq:abcf} is a coflabby resolution if and only if (a) is valid for all subgroups of prime power order $H$ of $G$. Finally, the equivalence between (a) and (b) follows from the vanishing of $H^{1}(H,R)$. 
\end{proof}

\begin{prop}\label{prop:cfcf}
The sequence \eqref{eq:cdcf} is a coflabby resolution of $I_{m}$. 
\end{prop}

\begin{proof}
It suffices to prove (a) or (b) in Lemma \ref{lem:flpp} holds for any subgroup $H$ of prime power order in $D_{2m}$. Hence, we may assume that $H$ is contained in one of the following: 
\begin{equation*}
\langle \sigma_{2m}^{m}\rangle,\quad \langle \tau_{2m}\rangle,\quad \langle \sigma_{2m}^{m}\tau_{2m}\rangle,\quad \langle \sigma_{2m}^{m},\tau_{2m}\rangle,\quad \langle \sigma_{2m}^{2}\rangle. 
\end{equation*}

\textbf{Case 1.~$H=\langle \tau_{2m}\rangle$ or $H\subset \langle \sigma_{2m}^{2}\rangle$. }

In this case, $H$ is contained in $D^{(m)}:=\langle \sigma_{2m}^{2},\tau_{2m}\rangle$. Hence, Lemma \ref{lem:rsdg} implies that $E_{m}$ is permutation as an $H$-lattice. In particular, we obtain $H^{1}(H,E_{m})=0$ as desired. 

\textbf{Case 2.~$H=\langle \sigma_{2m}^{m}\rangle$. }

By direct computation, we have $I_{m}^{H}=I_{m,0}+\Psi_{m}(R_{m,2})$. Hence, Lemma \ref{lem:zfsj} gives an equality
\begin{equation*}
    I_{m}^{H}=\Psi_{m}(R_{m,1}^{H}\oplus R_{m,2}\oplus R_{m,3}). 
\end{equation*}
Now, we obtain that (b) holds, that is, $R_{m}^{H}\rightarrow I_{m}^{H}$ is surjective. 

\textbf{Case 3.~$H=\langle \sigma_{2m}^{m},\tau_{2m}\rangle$. }

The equality $I_{m}^{H}=(1+\tau_{2m})(I_{m,0}+\Psi_{m}(R_{m,2}))+\Psi_{m}(R_{m,3})$ can be confirmed directly. In particular, one has
\begin{equation*}
    I_{m}^{H}=\Psi_{m}(R_{m,1}^{H}\oplus R_{m,2}^{H}\oplus R_{m,3}), 
\end{equation*}
which follows from Case 2. Hence (b) is valid. 

\textbf{Case 4.~$H=\langle \sigma_{2m}^{m}\tau_{2m}\rangle$. }

By the definition of $\Psi_{m}$, we obtain $I_{m}^{H}=(1+\sigma_{2m}^{m}\tau_{2m})(I_{m,0}+\Psi_{m}(R_{m,1}))$. In particular, we have
\begin{equation*}
    I_{m}^{H}=(1+\sigma_{2m}^{m}\tau_{2m})I_{m}. 
\end{equation*}
On the other hand, the homomorphism $R_{m}\rightarrow I_{m}$ is surjective, which is contained in Corollary \ref{cor:eris}. Therefore, we obtain that (b) is valid in this case.  
\end{proof}

\begin{prop}\label{prop:fmiv}
The $D_{2m}$-lattice $E_{m}$ is invertible. 
\end{prop}

\begin{proof}
By Proposition \ref{prop:cfcf}, it suffices to prove that the $D_{2m}$-lattice $J_{D_{2m}/\cH_{m}}$ is quasi-invertible. By Theorem \ref{thm:qips}, it is equivalent to $J_{D_{2m}/\cH_{m}}$ being quasi-invertible as an $S_{p}$-lattice for any $p\mid 2m$. Here, $S_{p}$ is a $p$-Sylow subgroup of $D_{2m}$. If $p$ is an odd prime, then $S_{p}$ is cyclic, and hence the assertion follows from Proposition \ref{prop:emqi}. Now, put $S:=\langle \sigma_{2m}^{m},\tau_{2m}\rangle$, which is a $2$-Sylow subgroup of $D_{2m}$. Then direct computation implies that the subsets $\{1,\sigma_{2m},\ldots,\sigma_{2m}^{m-1}\}$ and $\{1,\sigma_{2m},\ldots,\sigma_{2m}^{(m-1)/2}\}$ are complete representatives of $S\backslash D_{2m}/\langle \sigma_{2m}^{m}\rangle$ and $S\backslash D_{2m}/\langle \tau_{2m}\rangle$, respectively. Hence, there is an isomorphism of $S$-lattices $J_{D_{2m}/\cH_{m}}\cong J_{S/\cH_{0}}$, where
\begin{equation*}
\cH_{0}=\{\underbrace{\langle \sigma_{2m}^{m}\rangle,\ldots,\langle \sigma_{2m}^{m}\rangle}_{m},\langle \tau_{2m}\rangle,\underbrace{\{1\},\ldots,\{1\}}_{(m-1)/2}\}. 
\end{equation*}
In particular, one has an isomorphism of $S$-lattices
\begin{equation*}
J_{S/\cH_{0}}\cong J_{S/\{\langle \sigma_{2m}^{m}\rangle,\langle \tau_{2m}\rangle\}}\oplus \Z[S/\langle \sigma_{2m}^{m}\rangle]^{\oplus m-1}\oplus \Z[S]^{\oplus (m-1)/2}. 
\end{equation*}
On the other hand, the $S$-lattice $J_{S/\{\langle \sigma_{2m}^{m}\rangle,\langle \tau_{2m}\rangle\}}$ is quasi-invertible by Proposition \ref{prop:tgqp}. Therefore, we obtain that $J_{D_{2m}/\cH_{m}}$ is quasi-invertible as an $S$-lattice. This completes the proof by Theorem \ref{thm:qips}. 
\end{proof}

We now proceed to Step 3. 

\begin{lem}\label{lem:stfx}
Define two elements $y_{0}$ and $y_{1}$ of $R_{m}$ as follows: 
\begin{equation*}
y_{0}:=\left(\sum_{i=0}^{2m-1}\sigma_{2m}^{i},-\sum_{j=0}^{m-1}\sigma_{2m}^{2j},0\right),\quad y_{1}:=\left(\left(\sum_{j=0}^{m-1}\sigma_{2m}^{2j}\right)(1+\sigma_{2m}\tau_{2m}),0,-1\right). 
\end{equation*}
\begin{enumerate}
\item We have $y_{0}\in E_{m}^{\langle \sigma_{2m}\rangle}$ and $y_{1}\in E_{m}^{\langle \sigma_{2m}^{2},\sigma_{2m}\tau_{2m}\rangle}$. 
\item There exist isomorphisms 
\begin{equation*}
\langle y_{0}\rangle_{\Z[D_{2m}]}\cong \Z[D_{2m}/\langle \sigma_{2m}\rangle],\quad \langle y_{1}\rangle_{\Z[D_{2m}]}\cong \Z[D_{2m}/\langle \sigma_{2m}^{2},\sigma_{2m}\tau_{2m}\rangle]. 
\end{equation*}
\item One has $E_{m}^{\langle \sigma_{2m}^{2}\rangle}=\langle y_{0}\rangle_{\Z[D_{2m}]}\oplus \langle y_{1}\rangle_{\Z[D_{2m}]}$. In particular, there is an isomorphism of $D_{2m}$-lattices
\begin{equation*}
E_{m}^{\langle \sigma_{2m}^{2}\rangle}\cong \Z[D_{2m}/\langle \sigma_{2m}\rangle]\oplus \Z[D_{2m}/\langle \sigma_{2m}^{2},\sigma_{2m}\tau_{2m}\rangle]. 
\end{equation*}
\end{enumerate}
\end{lem}

\begin{proof}
(i): This follows from the definition of $\Psi_{m}$. 

(ii): By definition, the composite maps
\begin{gather*}
\Z[D_{2m}/\langle \sigma_{2m}\rangle]\xrightarrow{1\mapsto y_{0}} \langle y_{0}\rangle_{\Z[D_{2m}]}\xrightarrow{\pr_{1}}
\Z[D_{2m}]^{\langle \sigma_{2m}\rangle}\cong \Z[D_{2m}/\langle \sigma_{2m}\rangle],\\
\Z[D_{2m}/\langle \sigma_{2m}^{2},\sigma_{2m}\tau_{2m}\rangle]\xrightarrow{1\mapsto y_{1}} \langle y_{1}\rangle_{\Z[D_{2m}]}\xrightarrow{\pr_{1}}\Z[D_{2m}]^{\langle \sigma_{2m}^{2},\sigma_{2m}\tau_{2m}\rangle}\cong \Z[D_{2m}/\langle \sigma_{2m}^{2},\sigma_{2m}\tau_{2m}\rangle]
\end{gather*}
are identity maps. On the other hand, (i) implies $\rk_{\Z}\langle y_{0}\rangle_{\Z[D_{2m}]}=\rk_{\Z}\langle y_{1}\rangle_{\Z[D_{2m}]}=2$. Hence, the assertions hold. 

(iii): We may assume $m=1$, which is a consequence of Lemma \ref{lem:qtfx} and Proposition \ref{prop:cfcf}. By definition, the homomorphism
\begin{equation}\label{eq:exmo}
    R_{1,1}\oplus R_{1,3} \rightarrow I_{1};\,(x,y)\mapsto \psi_{1,1}(x)+\psi_{1,3}(y)
\end{equation}
is surjective, and its kernel coincides with $\langle y_{1},\sigma_{2}(y_{1})\rangle_{\Z}=\langle y_{1}\rangle_{\Z[D_{2}/\langle \sigma_{2}\tau_{2} \rangle]}$. This follows from the same argument as \cite[p.~30, \S 2]{Endo2001}. Hence, \eqref{eq:exmo} induces an exact sequence of $D_{2}$-lattices
\begin{equation*}
    0\rightarrow \Z[D_{2}/\langle \sigma_{2}\tau_{2}\rangle]\xrightarrow{1\mapsto y_{1}} R_{1,1}\oplus R_{1,3} \rightarrow I_{1}\rightarrow 0. 
\end{equation*}
This is a coflabby resolution of $I_{1}$, and hence we obtain the desired assertion without difficulty. 
\end{proof}

We recall the basic fact given by Lenstra (\cite{Lenstra1974}). Note that Step 2 is necessary in order to apply the following.

\begin{lem}[{\cite[(1.2) Proposition (c)]{Lenstra1974}}]\label{lem:lstr}
Let $G$ be a finite group, and $E$ an invertible $G$-lattice. For a coflabby $G$-lattice $F$, every exact sequence of $G$-lattices
\begin{equation*}
0\rightarrow F \rightarrow Q \rightarrow E \rightarrow 0
\end{equation*}
splits. 
\end{lem}

\begin{prop}\label{prop:emsp}
There is an isomorphism of $D_{2m}$-lattices
\begin{equation*}
E_{m}\oplus \Z[D_{2m}/\langle \sigma_{2m}^{2},\sigma_{2m}^{m}\tau_{2m}\rangle] \cong \Z[D_{2m}/\langle \sigma_{2m}^{m}\tau_{2m}\rangle]\oplus \Z[D_{2m}/\langle \sigma_{2m}\rangle]\oplus \Z[D_{2m}/\langle \sigma_{2m}^{2},\sigma_{2m}^{m}\tau_{2m}\rangle]. 
\end{equation*}
\end{prop}

\begin{proof}
Put $Q_{m}:=\Z[D_{2m}/\langle \sigma_{2m}^{m}\tau_{2m}\rangle]\oplus \Z[D_{2m}/\langle \sigma_{2m}\rangle] \oplus \Z[D_{2m}/\langle \sigma_{2m}^{2},\sigma_{2m}^{m}\tau_{2m}\rangle]$, and define a homomorphism of $D_{2m}$-lattices as follows: 
\begin{equation*}
\Xi_{m} \colon Q_{m} \rightarrow E_{m};\,(a,b,c)\mapsto ax_{1}+by_{0}+cy_{1}. 
\end{equation*}
By definition, one has $x_{0}=\sigma_{2m}(y_{1})-y_{0}$. Hence, Lemma \ref{lem:rsdg} implies that $\Xi_{m}$ is surjective. In particular, we obtain an exact sequence of $D_{2m}$-lattices
\begin{equation}\label{eq:ffde}
0\rightarrow F_{m} \rightarrow Q_{m} \xrightarrow{\Xi_{m}}E_{m} \rightarrow 0, 
\end{equation}
where $\rk_{\Z}(F_{m})=2$. On the other hand, consider the commutative diagram
\begin{equation*}
\xymatrix{
0\ar[r]&F_{m}^{\langle \sigma_{2m}^{2}\rangle}\ar[r]\ar[d]&
Q_{m}^{\langle \sigma_{2m}^{2}\rangle}\ar[r]^{\Xi_{m}}\ar[d]&
E_{m}^{\langle \sigma_{2m}^{2}\rangle}\ar[r]\ar[d]&0\\
0\ar[r]&F_{m}\ar[r]&Q_{m}\ar[r]^{\Xi_{m}}&E_{m}\ar[r]&0,
}
\end{equation*}
where the horizontal sequences are exact. Then we have $\rk_{\Z}(E_{m}^{\langle \sigma_{2m}^{2}\rangle})=4$ and $\rk_{\Z}(Q_{m}^{\langle \sigma_{2m}^{2}\rangle})=6$. In particular, we obtain $\rk_{\Z}(F_{m}^{\langle \sigma_{2m}^{2}\rangle})=2$. Consequently, we have 
\begin{equation}
F_{m}=F_{m}^{\langle \sigma_{2m}^{2}\rangle}\cong \Z[D_{2m}/\langle \sigma_{2m}^{2},\sigma_{2m}^{m}\tau_{2m}\rangle]
\end{equation}
by Lemma \ref{lem:lmis}. This implies that $F_{m}$ is coflabby. On the other hand, Proposition \ref{prop:fmiv} implies that $E_{m}$ is invertible. Therefore, \eqref{eq:ffde} splits by Lemma \ref{lem:lstr}, which is equivalent to the desired assertion. 
\end{proof}

Finally, we complete Step 4. Define $D_{2m}$-lattices $\widetilde{R}_{m}$ and $\widetilde{E}_{m}$ as follows: 
\begin{equation*}
    \widetilde{R}_{m}:=R_{m}\oplus \Z[D_{2m}/\langle \sigma_{2m}^{2},\sigma_{2m}^{m}\tau_{2m}\rangle],\quad
    \widetilde{E}_{m}:=E_{m}\oplus \Z[D_{2m}/\langle \sigma_{2m}^{2},\sigma_{2m}^{m}\tau_{2m}\rangle]. 
\end{equation*}
By Proposition \ref{prop:emsp}, there is an isomorphism of $D_{2m}$-lattices
\begin{equation*}
\widetilde{E}_{m}\cong \Z[D_{2m}/\langle \sigma_{2m}^{m}\tau_{2m}\rangle]\oplus \Z[D_{2m}/\langle \sigma_{2m}\rangle]\oplus \Z[D_{2m}/\langle \sigma_{2m}^{2},\sigma_{2m}^{m}\tau_{2m}\rangle]. 
\end{equation*}
In particular, $\widetilde{E}_{m}$ is permutation. Moreover, the coflabby resolution \eqref{eq:cdcf} induces an exact sequence
\begin{equation*}
    0\rightarrow \widetilde{E}_{m}\rightarrow \widetilde{R}_{m}\rightarrow I_{m}\rightarrow 0. 
\end{equation*}
Taking its dual, we obtain an exact sequence of $D_{2m}$-lattices
\begin{equation*}
    0\rightarrow J_{D_{2m}/\cH_{m}}\rightarrow \widetilde{R}_{m}^{*} \rightarrow \widetilde{E}_{m}^{*} \rightarrow 0. 
\end{equation*}
Since $ \widetilde{R}_{m}^{*}$ and $\widetilde{E}_{m}^{*}$ are permutation, we obtain that $J_{D_{2m}/\cH_{m}}$ is quasi-permutation. Hence, the proof of Theorem \ref{thm:nzf2} is complete. \qed

\section{Proof of Theorem \ref{mth2}}\label{sect:pfdn}

\begin{lem}\label{lem:rdos}
Let $p$ be a prime number. Put $G=G_{p}\rtimes G'$, where $G_{p}$ is a finite abelian $p$-group, and $G'$ is a finite group of order coprime to $p$. We further assume that
\begin{itemize}
\item $G_{p}\cong (C_{p^{m}})^{n}$ for some $m,n\in \Zpn$; and
\item $G'$ acts on $G_{p}$ by scalar multiplication. 
\end{itemize}
Consider a multiset $\cH$ of subgroups of $G$. Then there is a strongly reduced set $\cH_{0}$ of $G$ such that
\begin{enumerate}
\item $(\cH_{0})_{G_{p}}^{\set}=\cH_{G_{p}}^{\red}$; and
\item $[I_{G/\cH}]=[I_{G/\cH_{0}}]$ and $[J_{G/\cH}]=[J_{G/\cH_{0}}]$. 
\end{enumerate}
\end{lem}

\begin{proof}
By Lemma \ref{lem:cjsd} and Proposition \ref{prop:conj}, we may assume that all elements of $\cH$ are of the form $N\rtimes H'$, where $N<G_{p}$ and $H'<G'$. Since $G'$ acts on $G_{p}$ by scalar, we have
\begin{equation*}
G_{p}\cap gHg^{-1}=G_{p}\cap H
\end{equation*}
for any $H\in \cH$ and $g\in G$. Then the assertion follows from the same argument as \cite[Lemma 8.3]{Hasegawa}. 
\end{proof}

\begin{lem}\label{lem:dnwl}
Let $n$ be a positive integer, and $\cH$ a multiset of subgroups of $D_{n}$ with $d=d_{D_{n}}(\cH)$. Then there exists a strongly reduced set $\cH'$ of subgroups of $D_{n}$ that satisfies the following: 
\begin{enumerate}
\item $\cH'\cap \cH_{d}\neq \emptyset$ and $\left(\cH'\cup \cH_{d}\right)^{\red}=\cH_{d}$, where
\begin{equation*}
\cH_{d}:=
\begin{cases}
\{\langle \sigma_{n}^{d},\tau_{n}\rangle\}&\text{if }d \notin 2\Z;\\
\{\langle \sigma_{n}^{d/2}\rangle\}&\text{if $d\in 2\Z$ and }d\nmid n;\\
\{\langle \sigma_{n}^{d/2}\rangle,\langle \sigma_{n}^{d},\tau_{n}\rangle,\langle \sigma_{n}^{d},\sigma_{n}\tau_{n}\rangle\}&\text{if $d\in 2\Z$ and }d \mid n;
\end{cases}
\end{equation*}
\item $\#H/N^{D_{n}}(\cH')$ is a power of $2$ for any $H\in \cH'$; and
\item $[J_{D_{n}/\cH}]=[J_{D_{n}/\cH'}]$ in $\sS(D_{n})$. 
\end{enumerate}
\end{lem}

\begin{proof}
If $d=1$, then $\cH':=\{D_{n}\}$ satisfies (i) and (ii). Moreover, Corollary \ref{cor:gcdo} implies $[J_{D_{n}/\cH}]=0=[J_{D_{n}/\cH'}]$. Hence, we may further assume $d>1$ in the sequel. Now, take an odd prime divisor $p$ of $n$, and let $\mu_{p}:=\ord_{p}(n)$. Put $N_{p}:=\langle \sigma_{n}^{n/p^{\mu_{p}}}\rangle$, which is a $p$-Sylow subgroup of $D_{n}$. Then there is an isomorphism 
\begin{equation*}
D_{n}=N_{p} \rtimes \langle \sigma_{n}^{p^{\mu_{p}}},\tau_{n} \rangle. 
\end{equation*}
Hence, by Lemma \ref{lem:rdos}, one can assume that $\cH$ is strongly reduced and $\cH_{N_{p}}^{\set}$ is reduced for any odd prime divisor $p$ of $n$. This implies that all elements of $\cH$ contain $\sigma_{n}^{n/m}$, where $m$ is the maximum odd divisor of $2n/d$. Therefore, by replacing $D_{n}$ with its quotient by $\langle \sigma_{n}^{n/m}\rangle$, we may further assume that $\#H$ is a power of $2$ for any $\cH$. In particular, there exists $H_{0}\in \cH$ such that $(D_{n}:H_{0})=d_{D_{n}}(\cH)$.

Write $n=2^{\nu_{2}}n'$, where $\nu_{2} \in \Znn$ and $n'$ is the maximum odd divisor of $n$. Note that $n'$ divides $d$ and $d/n'$ is a power of $2$, since $\#H$ is a power of $2$ for any $H\in \cH$. For any $H\in \cH$, there is $g\in D_{n}$ such that $gHg^{-1}$ is one of the forms $\langle \sigma_{n}^{2^{\nu}n'}\rangle$, $\langle \sigma_{n}^{2^{\nu}n'},\tau_{n}\rangle$ or $\langle \sigma_{n}^{2^{\nu}n'},\sigma_{n}\tau_{n}\rangle$ for some $\nu \in \{0,\ldots,\nu_{2}\}$. Note that $\langle \sigma_{n}^{n'},\tau_{n}\rangle$ and $\langle \sigma_{n}^{n'},\sigma_{n}\tau_{n}\rangle$ are conjugate to each other. In particular, by Proposition \ref{prop:conj}, we may assume the following: 
\begin{equation*}
    \cH \subset \left(\bigcup_{\nu=0}^{\nu_{2}}\{\langle \sigma_{n}^{2^{\nu}n'}\rangle,\langle \sigma_{n}^{2^{\nu}n'},\tau_{n}\rangle,\langle \sigma_{n}^{2^{\nu}n'},\sigma_{n}\tau_{n}\rangle \}\right)\setminus \{\langle \sigma_{n}^{n'},\sigma_{n}\tau_{n}\rangle\}. 
\end{equation*}

\textbf{Case 1.~$d$ is odd. }

Since $(D_{n}:H_{0})=d=n'$, we have 
$\cH=\{\langle \sigma_{n}^{n'},\tau_{n}\rangle\}$. This satisfies (i) as desired. 

\textbf{Case 2.~$d$ is even. }

In this case, $H_{0}$ lies in $\cH_{d}$. If $d\nmid n$, then $\#H_{0}$ must be odd, and hence $H_{0}=\langle \sigma_{n}^{d/2}\rangle$. Therefore, we obtain $\cH=\cH_{d}$, which implies that (i) holds. On the other hand, if $d\mid n$, then the validity of (i) is clear. 
\end{proof}

Theorem \ref{mth2} follows from Proposition \ref{prop:coch} and Theorem \ref{thm:dihd} as follows. 

\begin{thm}\label{thm:dihd}
Let $n$ be a positive integer, and $\cH$ a multiset of subgroups of $D_{n}$. Write $d=d_{D_{n}}(\cH)$. 
\begin{enumerate}
\item If $d$ is odd, then the $D_{n}$-lattice $J_{D_{n}/\cH}$ is quasi-permutation. 
\item If $d$ is even, then the $D_{n}$-lattice $J_{D_{n}/\cH}$ is quasi-permutation if and only if it is quasi-invertible. Furthermore, the above two conditions are valid if and only if there exists a set of subgroups $\cH'$ of $D_{n}$ such that
\begin{equation*}
[J_{D_{n}/\cH}]=[J_{D_{n}/\cH'}], 
\end{equation*}
where $d$ and $\cH'$ satisfy one of the following: 
\begin{itemize}
\item[(a)] $d\in 2\Z \setminus 4\Z$, $d\mid 2n$ and $\cH'=\{\langle \sigma_{n}^{d/2}\rangle\}$; 
\item[(b)] $d\in 2\Z \setminus 4\Z$, $d\mid n$ and $\cH'=\{\langle \sigma_{n}^{d},\sigma_{n}^{j}\tau_{n}\rangle\}$ for some $j\in \{0,1\}$; 
\item[(c)] $d\in 2\Z \setminus 4\Z$, $d\mid n$ and $\cH'=\{\langle \sigma_{n}^{d/2} \rangle,\langle \sigma_{n}^{d},\sigma_{n}^{j}\tau_{n} \rangle\}$ for some $j\in \{0,1\}$; or
\item[(d)] $d\mid n$ and $\cH'=\{\langle \sigma_{n}^{d},\tau_{n} \rangle,\langle \sigma_{n}^{d},\sigma_{n}\tau_{n} \rangle\}$.  
\end{itemize}
\end{enumerate}
\end{thm}

\begin{proof}
(i): We may assume that $\cH$ satisfies (i) and (ii) in Lemma \ref{lem:dnwl}. Then we have 
\begin{equation*}
    \cH=\{\langle \sigma_{n}^{d},\tau_{n}\rangle\}. 
\end{equation*}
Hence, using Corollary \ref{cor:trfx}, we obtain that $D_{n}/\langle \sigma_{n}^{d}\rangle$ induces an isomorphism of $D_{n}$-lattices $J_{D_{n}/\cH}\cong J_{D_{d}/\langle \tau_{d}\rangle}$. Hence, the assertion follows from Theorem \ref{thm:emmn}. 

(ii): We first prove that $J_{D_{n}/\cH}$ is quasi-permutation if $\cH$ satisfies (a), (b), (c) or (d). If (a) is valid, then the canonical isomorphism $D_{n}/\langle \sigma_{n}^{d/2}\rangle \cong D_{d/2}$ induces an isomorphism $J_{D_{n}/\cH}\cong J_{D_{d/2}}$. Hence the assertion follows from Theorem \ref{thm:emmn}. For the remaining cases, we may assume $d=n$, which follows from Corollary \ref{cor:trfx}. If (b) and $d\geq 6$ hold, then the assertion is a consequence of \cite[Theorem 1.1 (ii)]{Hoshi2024}. If (b) and $d=2$ are valid, then we have $D_{n}/N^{D_{n}}(\cH)\cong C_{2}$. Moreover, this induces an isomorphism of $C_{2}$-lattices $J_{D_{n}/\cH}\cong J_{C_{2}}$. Therefore, we obtain the desired claim by Theorem \ref{thm:emmn}. If (c) is satisfied, then Theorem \ref{thm:nzf2} gives the desired assertion. Finally, if (d) holds, then the assertion is the same as \cite[Theorem 6.1]{Hasegawa}. 

Next, assume that $J_{D_{n}/\cH}$ is quasi-invertible. 
We may assume that $\cH$ is a strongly reduced set and satisfies (i) and (ii) in Lemma \ref{lem:dnwl}. Write $n=2^{\nu}n'$, where $\nu \in \Znn$ and $n'\in \Zpn \setminus 2\Z$. Put $N:=\langle \sigma_{n}^{2^{\nu}}\rangle$. Then $N$ is normal in $D_{n}$, and $\overline{D}_{n}:=D_{n}/N$ is isomorphic to the $2$-group $D_{2^{\nu}}$. Moreover, there is an isomorphism of $D_{n}/N$-lattices
\begin{equation*}
J_{D_{n}/\cH}^{[N]}\cong J_{\overline{D}_{n}/\overline{\cH}}^{(\overline{\varphi}^{\nor})}.  
\end{equation*}
Here $\overline{\cH}:=\{HN/N\mid H\in \cH\}$ is a strongly reduced set, and $\overline{\varphi}^{\nor}(\overline{H})=1$ for all $\overline{H}\in \overline{\cH}$. This is a consequence of Proposition \ref{prop:tkfx}. In particular, we have $J_{\overline{D}_{n}/\overline{\cH}}^{(\overline{\varphi}^{\nor})}=J_{\overline{D}_{n}/\overline{\cH}}$. Applying Proposition \ref{prop:tgqp}, we obtain that $\overline{\cH}$ satisfies 
\begin{itemize}
\item[(A)] $\#\overline{\cH}=1$ and $\overline{D}_{n}/N^{\overline{D}_{n}}(\overline{\cH})$ is cyclic; 
\item[(B)] $\#\overline{\cH}=2$ and $\overline{D}_{n}/N^{\overline{D}_{n}}(\overline{\cH}) \cong C_{2}^{2}$; or
\item[(C)] $\overline{\cH}=\{H,H'\}$ and there exist $\nu \in \Z$ with $\nu \geq 2$ and an isomorphism $\overline{D}_{n}/N^{\overline{D}_{n}}(\cH) \cong D_{2^{\nu}}$ that induces $H/N^{\overline{D}_{n}}(\cH)\cong \langle \sigma_{2^{\nu}}^{2m}\tau_{2^{\nu}} \rangle$ and $H'/{N^{\overline{D}_{n}}(\cH)}\cong \langle \sigma_{2^{\nu}}^{2m'+1}\tau_{2^{\nu}} \rangle$ for some $m,m'\in \Z$. 
\end{itemize}
If (A) holds, then we must have (a) or (b). If (B) is valid, then $\cH$ satisfies (c) or (d). Finally, if (C) is satisfied, then we obtain the validity of (d). Hence, the proof is complete. 
\end{proof}

\begin{rem}
Assume $n$ is even and $d$ is a divisor of $n$ with $d\in 2\Z \setminus 4\Z$. Then the map $f\colon D_{n}\rightarrow D_{n}$ given by $f(\sigma_{n})=\sigma_{n}$ and $f(\tau_{n})=\sigma_{n}\tau_{n}$ is an automorphism. In particular, it induces isomorphisms of $D_{n}$-lattices
\begin{equation*}
    J_{D_{n}/\langle \sigma_{n}^{d},\tau_{n}\rangle}\cong f^{*}J_{D_{n}/\langle \sigma_{n}^{d},\sigma_{n}\tau_{n}\rangle},\quad 
    J_{D_{n}/\{\langle \sigma_{n}^{d/2} \rangle,\langle \sigma_{n}^{d},\tau_{n} \rangle\}}\cong 
    f^{*}J_{D_{n}/\{\langle \sigma_{n}^{d/2} \rangle,\langle \sigma_{n}^{d},\sigma_{n}\tau_{n} \rangle\}}. 
\end{equation*}
Therefore, (b) (resp.~(c)) for $j=1$ is reduced to that for $j=0$. 
\end{rem}

\end{document}